\documentclass[11pt]{article}
\usepackage{amsmath}
\usepackage{amsthm}
\usepackage{amsfonts}
\usepackage{amssymb}
\usepackage{amscd}
\usepackage[dvips]{graphicx, color}

\newtheorem{thm}{Theorem}[section]
\newtheorem{cor}[thm]{Corollary}

\newtheorem{lem}[thm]{Lemma}
\newtheorem{prop}[thm]{Proposition}
\newtheorem{rem}[thm]{Remark}
\newtheorem{exam}[thm]{Example}

\numberwithin{equation}{section}

\newcommand{\Mod}{\text{\boldmath{$\mathsf{M}$}}}

\begin{document}

\title{\bf \Large 
Schr\"{o}dinger representations from the viewpoint of \linebreak monoidal categories} 
\author{Kenichi Shimizu and Michihisa Wakui}
\date{}
\maketitle 

\begin{abstract}
  The Drinfel'd double $D(A)$ of a finite-dimensional Hopf algebra $A$ is a Hopf algebraic counterpart of the monoidal center construction. 
Majid introduced an important representation of $D(A)$, which he called the Schr\"{o}dinger representation. 
We study this representation from the viewpoint of the theory of monoidal categories. One of our main results is as follows: 
If two finite-dimensional Hopf algebras $A$ and $B$ over a field $\boldsymbol{k}$ are monoidally Morita equivalent, {\it i.e.}, there exists an equivalence $F: {}_A\Mod \to {}_B\Mod$ of $\boldsymbol{k}$-linear monoidal categories, 
then the equivalence ${}_{D(A)}\Mod \approx {}_{D(B)}\Mod$ induced by $F$ preserves the Schr\"{o}dinger representation. 
Here, ${}_A\Mod$ for an algebra $A$ means the category of left $A$-modules.

As an application, we construct a family of invariants of finite-dimensional Hopf algebras under the monoidal Morita equivalence. 
This family is parameterized by braids. 
The invariant associated to a braid $\boldsymbol{b}$ is, roughly speaking, defined by \lq\lq coloring'' the closure of $\boldsymbol{b}$ by the Schr\"{o}dinger representation. 
We investigate what algebraic properties this family have and, in particular, 
show that the invariant associated to a certain braid closely relates to the number of irreducible representations.
\end{abstract}


\baselineskip 17pt

\section{Introduction}

Drinfel'd doubles of Hopf algebras \cite{Dri} are one of the most important objects in not only Hopf algebra theory, 
but also in other areas including category theory and low-dimensional topology. 
Let $A$ be a finite-dimensional Hopf algebra over a field $\boldsymbol{k}$, and $(D(A), \mathcal{R})$ be its Drinfel'd double. 
Due to Majid \cite{MajidBook}, it is known that there is a canonical representation of $D(A)$ on $A$, 
which is called the {\em Schr\"{o}dinger representation} (or the {\em Schr\"odinger module}). 
This representation is an extension of the adjoint representation of $A$, and originates from quantum mechanics as explained in Majid's book \cite[Examples 6.1.4 \& 7.1.8]{MajidBook} (see Section 2 for the precise definition of the Schr\"{o}dinger representation). 
The Schr\"odinger module is also addressed by Fang \cite{Fang} as an algebra in the Yetter-Drinfel'd category ${}^A_A\mathcal{YD}$ via the Miyashita-Ulbrich action.

In this paper, we study the Schr\"odinger module over the Drinfel'd double from the viewpoint of the theory of monoidal categories. 
We say that two finite-dimensional Hopf algebras $A$ and $B$ over the same field $\boldsymbol{k}$ are {\em monoidally Morita equivalent} if ${}_A\Mod$ and ${}_B\Mod$ are equivalent as $\boldsymbol{k}$-linear monoidal categories, 
where ${}_H \Mod$ for an algebra $H$ is the category of $H$-modules. 
One of our main results is that the Schr\"odinger module is an invariant under the monoidal Morita equivalence in the following sense: 
If $F: {}_A\Mod \to {}_B\Mod$ is an equivalence of $\boldsymbol{k}$-linear monoidal categories, then the equivalence ${}_{D(A)}\Mod \approx {}_{D(B)}\Mod$ induced by $F$ preserves the Schr\"{o}dinger modules. 
To prove this result, we introduce the notion of the {\em Schr\"{o}dinger object} for a monoidal category by using the monoidal center construction. 
It turns out that the Schr\"odinger module over $D(A)$ is characterized as the Schr\"{o}dinger object for ${}_A\Mod$. 
Once such a characterization is established, the above result easily follows from general arguments.

As an application of the above category-theoretical understanding of the Schr\"{o}dinger module, 
we construct a new family of {\em monoidal Morita invariants}, {\it i.e.}, invariants of finite-dimensional Hopf algebras under the monoidal Morita equivalence. 
Some monoidal Morita invariants have been discovered and studied; see, {\it e.g.}, \cite{EG1, EG2, EG3, KMN, NS1, NS2, Shimizu, W}. 
Our family of invariants is parametrized by braids. 
Roughly speaking, the invariant associated with a braid $\boldsymbol{b}$ is defined by \lq\lq coloring'' the closure of $\boldsymbol{b}$ by the Schr\"odinger module. 
Since the quantum dimension (in the sense of Majid \cite{MajidBook}) of the Schr\"{o}dinger module is a special case of our invariants, we call the invariant associated with $\boldsymbol{b}$ the {\em braided dimension} of the Schr\"{o}dinger module associated with $\boldsymbol{b}$. 
We investigate what algebraic properties this family of invariants have and, in particular, 
show that the invariant associated to a certain braid closely relates to the number of irreducible representations.

This paper organizes as follows:
In Section 2,  we recall the definition of the Schr\"{o}dinger module over the Drinfel'd double $D(A)$ of a finite-dimensional Hopf algebra $A$.  
We also describe the definition of another Schr\"{o}dinger representation of $D(A)$ on $A^{\ast{\rm cop}}$, which is introduced by Fang \cite{Fang}.  
We refer the corresponding left $D(A)$-module as the {\em dual Schr\"{o}dinger module}. 
Following \cite{HuZhang} we also describe the definition of Radford's induction functors and its properties. 
It is shown that the Schr\"{o}dinger module and the dual Sch\"{o}dinger module are isomorphic to the images of the trivial left $A$-module and the trivial right $A$-comodule under Radford's induction functors, respectively. 
Furthermore, we examine the relationship between the Schr\"{o}dinger module over $D(A^{\ast})$ and the dual Schr\"{o}dinger module over $D(A)$.

In Section 3,  we study the categorical aspects of the Schr\"{o}dinger module and the dual Schr\"{o}dinger module. 
We introduce a Schr\"{o}dinger object for a monoidal category $\mathcal{C}$ as the object of $\mathcal{Z}(\mathcal{C})$ representing the functor ${\rm Hom}_{\mathcal{C}}(\Pi(-), \mathbb{I})$, 
where $\mathcal{Z}(\mathcal{C})$ is the monoidal center of $\mathcal{C}$, $\Pi: \mathcal{Z}(\mathcal{C}) \to \mathcal{C}$ is the forgetful functor, 
and $\mathbb{I}$ is the unit object of $\mathcal{C}$. 
By using the properties of Radford's induction functor, we show that the Schr\"{o}dinger module over $D(A)$ is a Schr\"odinger object for ${}_A \Mod$ under the identification $\mathcal{Z}({}_A \Mod) \approx {}_{D(A)}\Mod$. 
Once this characterization is obtained, the invariance of the Schr\"odinger module (stated above) is easily proved. 
A similar result for the dual Schr\"odinger module is also proved.

In Section 4, based on our category-theoretical understanding of the Schr\"{o}dinger module, 
we introduce a family of monoidal Morita invariants parameterized by braids. 
We give formulas for the invariants associated with a certain series of braids, and give some applications. 
Note that some monoidal Morita invariants, such as ones introduced in \cite{EG1} and \cite{Shimizu}, factor through the Drinfel'd double construction. 
Our invariants have an advantage that they do not factor through that. 
On the other hand, our invariants have a disadvantage in the non-semisimple situation: For any braid $\boldsymbol{b}$, the braided dimension of the Schr\"{o}dinger module of $D(A)$ associated with $\boldsymbol{b}$ is zero unless $A$ is cosemisimple. 
From this result, we could say that our invariants are not interesting as monoidal Morita invariants for non-cosemisimple Hopf algebras. 
However, endomorphisms on the Sch\"{o}dinger module induced by braids are not generally zero, and thus may have some information about $A$. 
To demonstrate, we give an example of a morphism induced by a braid, which turns out to be closely related to the unimodularlity of $A$.

\par 
Throughout this paper, $\boldsymbol{k}$ is an arbitrary field, and all tensor products $\otimes $ are taken over $\boldsymbol{k}$. 
For $\boldsymbol{k}$-vector spaces $X$ and $Y$, denoted by $T_{X, Y}$ is the $\boldsymbol{k}$-linear isomorphism from $X\otimes Y$ to $Y\otimes X$ defined by $T_{X,Y}(x\otimes y)=y\otimes x$ for all $x\in X$ and $y\in Y$. 
For a coalgebra $(C, \Delta , \varepsilon )$ and $c\in C$ we use the Sweedler's notation $\Delta (c)=\sum c_{(1)}\otimes c_{(2)}$. 
If $(M, \rho )$ is a right $C$-comodule and $m\in M$, then we also use the notation $\rho (m)=\sum m_{(0)}\otimes m_{(1)}$. 
For a Hopf algebra $A$, denoted by $\Delta _A$, $\varepsilon _A$ and $S_A$ are 
the coproduct, the counit, and the antipode of $A$,  respectively. 
If the antipode of $A$ is bijective, then one has two Hopf algebras $A^{{\rm cop}}$ and $A^{{\rm op}}$, which are defined from $A$ by replacing $\Delta_A$ by the  opposite coproduct $\Delta _A^{{\rm cop}}=T_{A,A}\circ \Delta _A$ and replacing the product by the opposite product, respectively. 
A Hopf algebra map $\alpha : A \longrightarrow B$ induces a Hopf algebra map form $A^{\rm cop}$ to $B^{\rm cop}$. 
The map is denoted by $\alpha^{\rm cop}$. 
If  $A$ is a finite-dimensional, then the dual vector space $A^{\ast}$ is also a Hopf algebra so that 
$\langle pq,\ a\rangle =\sum p(a_{(1)})q(a_{(2)}), 
1_{A^{\ast}}=\varepsilon _A, 
\langle \Delta _{A^{\ast}}(p),\ a\otimes b\rangle =p(ab),\ 
\varepsilon _{A^{\ast}}(p)=p(1_A),\ 
S_{A^{\ast}}(p)=p\circ S_A$ for 
$p,q\in A^{\ast}$ and $a,b\in A$, where $\langle \ , \ \rangle $ stands for the natural pairing between $A$ and $A^{\ast}$, or $A\otimes A$ and $A^{\ast}\otimes A^{\ast}$. 
Denoted by ${}_A\Mod$ is the $\boldsymbol{k}$-linear monoidal category whose objects are left $A$-modules and morphisms are $A$-module maps, and 
$\Mod^A$ is the $\boldsymbol{k}$-linear monoidal category whose objects are right $A$-comodules and morphisms are $A$-comodule maps. 

\par 
For general facts on Hopf algebras, refer to Abe's book \cite{Abe}, Montgomery's book \cite{Mon} and Sweedler's book \cite{Sw}, and for general facts on monoidal categories, refer to MacLane's book \cite{Mac}, Kassel's book \cite{Kassel} and Joyal and Street's paper \cite{JS}. 

\section{Schr\"{o}dinger modules of the Drinfel'd double}
\par 

\par \smallskip 
\subsection{Preliminaries: the Drinfel'd double}

Let $A$ be a finite-dimensional Hopf algebra $A$ over the field $\boldsymbol{k}$. 
The Drinfel'd double $D(A)$ of $A$ \cite{Dri} is the Hopf algebra such that 
as a coalgebra $D(A)=A^{\ast{\rm cop}}\otimes A$, and 
the multiplication is given by 
$$(p\otimes   a)\cdot (p^{\prime}\otimes  a^{\prime})
=\sum \langle p_{(1)}^{\prime}, S_A^{-1}(a_{(3)})\rangle 
\langle p_{(3)}^{\prime}, a_{(1)}\rangle 
p p_{(2)}^{\prime}\otimes a_{(2)}a^{\prime}$$ 
for all $p, p^{\prime}\in A^{\ast}$ and $a, a^{\prime}\in A$, where 
$((\Delta _A\otimes {\rm id})\circ \Delta _A)(a)=\sum a_{(1)}\otimes a_{(2)}\otimes a_{(3)},\ 
((\Delta _{A^{\ast}}\otimes {\rm id})\circ \Delta _{A^{\ast}})(p)=\sum p_{(1)}\otimes p_{(2)}\otimes p_{(3)}$.
The unit element is $\varepsilon _A\otimes 1_A$, and the antipode is given by 
$$S_{D(A)}(p\otimes  a)=\sum \langle p_{(1)}, a_{(3)}\rangle \langle S_{A^{\ast}}^{-1}(p_{(3)}), a_{(1)} \rangle S_{A^{\ast}}^{-1}(p_{(2)})\otimes  S_A(a_{(2)}).$$

For all $p\in A^{\ast}$ and $a\in A$,  the element $p\otimes a$ in $D(A)$ is frequently written in the form $p\bowtie a$, since the Drinfel'd double can be viewed as a bicrossed product $A^{\ast{\rm cop}}\bowtie A$ due to Majid \cite{MajidPhysics}.  
The Drinfel'd double $D(A)$ has a canonical quasitriangular structure as described below \cite{Dri}. 
Let $\{e_i\}_{i=1}^d$ be a basis for $A$, and $\{e_i^{\ast}\}_{i=1}^d$ be its dual basis for $A^{\ast}$.  Then 
$\mathcal{R}=\sum_{i=1}^d (\varepsilon _A\bowtie  e_i)\otimes (e_i^{\ast}\bowtie  1_A) \in D(A)\otimes D(A)$ satisfies 
the following conditions. 
\begin{enumerate}
\item[$\bullet$] $\Delta ^{\textrm{cop}}(x)=\mathcal{R}\cdot \Delta (x)\cdot \mathcal{R}^{-1} \ \ \mbox{ for all } x\in D(A)$, 
\item[$\bullet$] $(\Delta \otimes \textrm{id})(\mathcal{R})=\mathcal{R}_{13}\mathcal{R}_{23}$, 
\item[$\bullet$] $(\textrm{id}\otimes \Delta )(\mathcal{R})=\mathcal{R}_{13}\mathcal{R}_{12}$. 
\end{enumerate}
Here $\mathcal{R}_{ij}$ denotes the element in $D(A)\otimes D(A)\otimes D(A)$ obtained by substituting the first and the second components in 
$\mathcal{R}$ to the $i$-th and $j$-th components, respectively, and substituting $1$ to elsewhere. 
Thus, the pair $(D(A), \mathcal{R})$ is a quasitriangular Hopf algebra \cite{Dri}. 
\par 
  Let $A$ and $B$ be two finite-dimensional Hopf algebras over $\boldsymbol{k}$. If they are isomorphic, then their Drinfel'd doubles are. More precisely, an isomorphism $f: A \longrightarrow B$ of Hopf algebras yields an isomorphism of quasitriangular Hopf algebras
  \begin{equation*}
    D(f) := (f^{-1})^{\ast} \otimes f : D(A) \longrightarrow D(B).
  \end{equation*}

Let $A$ be a finite-dimensional Hopf algebra over $\boldsymbol{k}$.
As shown in \cite[Theorem 3]{Ra4} and \cite[Propositions 2.10 \& 3.5]{Shimizu}, 
for the Drinfel'd doubles $(D(A), \mathcal{R})$, $(D(A^{{\rm op\kern0.1em cop}\ast }), \tilde{\mathcal{R}})$ and 
$(D(A^{\ast}), \mathcal{R}^{\prime})$, there are isomorphisms: 
\begin{equation}\label{eq2-11}
(D(A), \mathcal{R})\cong (D(A^{{\rm op\kern0.1em cop}\ast })^{{\rm op}}, \tilde{\mathcal{R}}_{21})
\cong (D(A^{\ast})^{{\rm op}}, \mathcal{R}_{21}^{\prime})\cong (D(A^{\ast})^{{\rm cop}}, \mathcal{R}_{21}^{\prime}).
\end{equation} 
Here, the first isomorphism $f_1: (D(A), \mathcal{R})\longrightarrow (D(A^{{\rm op\kern0.1em cop}\ast })^{{\rm op}}, \tilde{\mathcal{R}}_{21})$ is given by 
$$f_1(p\bowtie a)=a\bowtie p\qquad (p\in A^{\ast},\ a\in A)$$
under the identification $A^{\ast\ast}=A$. 
The second isomorphism is given by $f_2:=S\otimes  {}^t\kern-0.1em S^{-1} : (D(A^{{\rm op\kern0.1em cop}\ast })^{{\rm op}}, \tilde{\mathcal{R}}_{21})
\longrightarrow (D(A^{\ast})^{{\rm op}}, \mathcal{R}_{21}^{\prime})$ under the identification $A^{\ast\ast}=A$, where 
we regard $S$ as a Hopf algebra isomorphism from  $A^{{\rm op\kern0.1em cop}}$ to $A$. 
Finally, the third isomorphism is given by the antipode $S_{D(A^{\ast})}$ \cite[Proposition 2.10(1)]{Shimizu}. 
Composing $f_1, f_2$ and $S_{D(A^{\ast})}$ we have an  isomorphism of quasitriangular Hopf algebras 
$\phi _A : (D(A), \mathcal{R})\longrightarrow (D(A^{\ast})^{{\rm cop}}, \mathcal{R}_{21}^{\prime})$ such that
\begin{equation}\label{eq2-12}
\phi _A(p\bowtie a)=(\iota _A(1_A)\bowtie p)\cdot (\iota _A(a)\bowtie \varepsilon _A)
\end{equation}
for all $p\in A^{\ast}$ and $a\in A$, where $\iota _A: A\longrightarrow A^{\ast\ast}$ is the usual isomorphism of vector spaces.
We note the following property of the isomorphism $\phi_A$:

  \begin{lem}
    \label{lem:Drinfeld-double-iso}
    $\phi_{A^{\ast}} \circ \phi_A = D(\iota_A)$.
  \end{lem}
  \begin{proof}
    As is well-known, $\iota_{A^{\ast}} = (\iota_A^{-1})^{\ast}$. Hence, for all $a \in A$ and $p \in A^{\ast}$,
    \begin{align*}
      (\phi_{A^{\ast}}\circ \phi_A)(p \bowtie a)
      & = \phi_{A^{\ast}}(\iota _A(1_A) \bowtie p) \cdot \phi_{A^{\ast}}(\iota_A(a) \bowtie \varepsilon_A) \\
      & = (\iota_{A^{\ast}}(p) \bowtie \iota_A(1_A)) \cdot (\iota_{A^{\ast}}(\varepsilon_A) \bowtie \iota_A(a)) = D(\iota_A)(p \bowtie a).
      \qedhere
    \end{align*}
  \end{proof}

\par \smallskip 
\subsection{Schr\"{o}dinger modules}

The Drinfel'd double $D(A)$ has a canonical representation, which is called the Schr\"{o}dinger representation as described in Majid's book \cite[Examples 6.1.4 \& 7.1.8]{MajidBook}. 
This representation is obtained by unifying the left adjoint action of $A$ and the right $A^{\ast}$-action $\leftharpoonup $ (See below (\ref{eq2-1}) and (\ref{eq2-2}) for precise definition). 
It is generalized to quasi-Hopf case by Bulacu and Torrecillas \cite[Section 3]{BulacuTorrecillas}.  
Another generalization is given by Fang \cite[Section 2]{Fang}. 
He introduced two Schr\"{o}dinger representations for the Drinfel'd double defined by a generalized Hopf pairing. 
One corresponds to the original Schr\"{o}dinger representation, and the other corresponds to a dual version of it. 
Specializing in our setting, we will describe these representations below. 
\par 
There are four actions defined as follows. 
\begin{align}
a\blacktriangleright c & =\sum a_{(1)}cS(a_{(2)}) \qquad (a,c\in A), \label{eq2-1}\\ 
a\leftharpoonup  p & =\sum \langle p , a_{(1)}\rangle a_{(2)} \qquad (p\in A^{\ast},\ a\in A), \label{eq2-2}\\ 
q \blacktriangleleft p  & = \sum S(p_{(1)})qp_{(2)} \qquad (p, q\in A^{\ast}) \label{eq2-3},\\ 
a\rightharpoonup  q & =\sum q_{(1)} \langle q_{(2)}, a\rangle \qquad (q\in A^{\ast},\ a\in A). \label{eq2-4}
\end{align}

By using these actions two left actions $\bullet $ of $D(A)$ on $A$ and $A^{\ast}$ can be defined by 
\begin{align}
(p\bowtie a) \bullet b  
& = (a\blacktriangleright b)\leftharpoonup S^{-1}(p), \label{eq2-5}\\ 
(p\bowtie a) \bullet q 
& = (a \rightharpoonup q) \blacktriangleleft S^{-1}(p)\label{eq2-6}
\end{align}
\noindent 
for all $a, b\in A^{\ast}$ and $p,q\in A^{\ast}$. 
We call the left $D(A)$-modules $(A, \bullet)$ and $(A^{\ast}, \bullet)$ the {\em Schr\"{o}dinger module} and the {\em dual Schr\"{o}dinger module}, and denote them by ${}_{D(A)}A$ and ${}_{D(A)}A^{\ast{\rm cop}}$, respectively.

\begin{prop}\label{2-7}
Regard  the dual Schr\"{o}dinger module ${}_{D(A^{\ast})}(A^{\ast})^{\ast{\rm cop}}$ of $D(A^*)$ as a left $D(A)$-module by pull-back along the isomorphism $\phi _A$ defined by (\ref{eq2-12}). 
Then 
${}_{D(A^{\ast})}(A^{\ast})^{\ast{\rm cop}}$ is isomorphic to the Schr\"{o}dinger module ${}_{D(A)}A$ as a left $D(A)$-module. 
\end{prop} 
\begin{proof}
Let $\bar{ \bullet}$ be the left action on ${}_{D(A^{\ast})}(A^{\ast})^{\ast{\rm cop}}=A$ as a left $D(A)$-module by pull-back along $\phi _A$. 
Then, for $p\in A^{\ast}$ and $a,b\in A$, we have 
\begin{align*}
(p\bowtie a)\bar{\ \bullet \ } b 
&=\phi _A(p\bowtie a)\bullet b \\ 
&=(1_A\bowtie p)\bullet \bigl( (a\bowtie \varepsilon _A)\bullet b\bigr) \\ 
&=(1_A\bowtie p)\bullet \bigl( \sum a_{(2)}bS^{-1}(a_{(1)})\bigr) \\ 
&=\sum \langle a_{(4)}b_{(2)} S^{-1}(a_{(1)}),\ p\rangle a_{(3)}b_{(1)}S^{-1}(a_{(2)}) \\ 
&=\sum \langle a_{(1)}S(b_{(2)})S(a_{(4)}),\ S^{-1}(p)\rangle S^{-1}\bigl( a_{(2)}S(b_{(1)})S(a_{(3)})\bigr) \\ 
&=S^{-1}\bigl( (p\bowtie a) \bullet S(b)\bigr) . 
\end{align*}
This implies that 
$S\bigl( (p\bowtie a)\bar{\ \bullet\ } b\bigr) =(p\bowtie a) \bullet S(b)$. 
Thus, the composition $S\circ \iota _A^{-1}: A^{\ast\ast} \longrightarrow A$ gives an isomorphism ${}_{D(A^{\ast})}(A^{\ast})^{\ast{\rm cop}}\longrightarrow {}_{D(A)}A$ of left $D(A)$-modules. 
\end{proof}

\begin{cor}\label{2-8}
  Regard the dual Schr\"{o}dinger module ${}_{D(A)}(A)^{\ast{\rm cop}}$ as a left $D(A^{\ast})$-module by pull-back along the isomorphism 
$(\phi_A^{-1})^{\rm cop} : (D(A^{\ast}), \mathcal{R}^{\prime})\longrightarrow (D(A)^{{\rm cop}}, \mathcal{R}_{21})$. 
Then 
${}_{D(A)}(A)^{\ast{\rm cop}}$ is isomorphic to the Schr\"{o}dinger module ${}_{D(A^{\ast})}A^{\ast}$ as a left $D(A^{\ast})$-module. 
\end{cor}
\begin{proof}
  Let $F_1: {}_{D(A)}\Mod \longrightarrow {}_{D(A^{**})}\Mod$ and $F_2: {}_{D(A)}\Mod \longrightarrow {}_{D(A^*)}\Mod$ be the isomorphisms of categories induced by isomorphisms $D(\iota_A^{-1})$ and $\phi_{A^*}$, respectively. By Lemma~\ref{lem:Drinfeld-double-iso}, $F = F_2\circ F_1$ is the isomorphism of categories induced by $\phi_{A}^{-1}$. Applying Proposition~\ref{2-7} to $A^*$, we have
  \begin{equation*}
    F({}_{D(A)}(A)^{\ast{\rm cop}})
    = (F_2\circ F_1)({}_{D(A)}(A)^{\ast{\rm cop}})
    = F_2({}_{D(A^{**})}(A^{**})^{\ast{\rm cop}})
    = {}_{D(A^{\ast})}A^{\ast}. \qedhere
  \end{equation*}
\end{proof}

\par \smallskip 
\subsection{Radford's induction functors and Schr\"{o}dinger modules} 
In this subsection we show that the Schr\"{o}dinger module ${}_{D(A)}A$ is isomorphic to the image of the trivial $A$-module under Radford's induction functor. 
This implies that the Schr\"{o}dinger modules are remarkable objects from the viewpoint of category theory. 
\par 
Radford's induction functors are described  by using Yetter-Drinfel'd modules, which were introduced in \cite{Yetter} and called crossed bimodules. 
Let us recall the definition of Yetter-Drinfel'd modules \cite{RadfordTowber}. 
Let $A$ be a bialgebra over $\boldsymbol{k}$.
Suppose that $M$ is a vector space over $\boldsymbol{k}$ equipped with a left $A$-module action $\cdot $ and a right $A$-comodule coaction $\rho$ on it. 
The triple $M=(M, \cdot, \rho )$ is called \textit{a Yetter-Drinfel'd $A$-module}, 
if the following compatibility condition, that is called the Yetter-Drinfel'd condition, is satisfied for all $a\in A$ and  $m\in M$: 
\begin{equation}\label{eq2-7}
\sum (a_{(1)}\cdot m_{(0)})\otimes (a_{(2)}m_{(1)})
=\sum (a_{(2)}\cdot m)_{(0)}\otimes (a_{(2)}\cdot m)_{(1)}a_{(1)}. 
\end{equation}
 A $\boldsymbol{k}$-linear map $f: M\longrightarrow N$ between two Yetter-Drinfel'd $A$-modules is called \textit{a Yetter-Drinfel'd homomorphism} if $f$ is an $A$-module map and an $A$-comodule map. 
For two Yetter-Drinfel'd $A$-modules $M$ and $N$, 
the tensor product $M\otimes N$ becomes a Yetter-Drinfel'd $A$-module with the action and coaction: 
\begin{align*}
a\cdot (m\otimes n) = \sum (a_{(1)}\cdot  m)\otimes (a_{(2)}\cdot n),\\ 
m\otimes n \longmapsto \sum m_{(0)}\otimes n_{(0)}\otimes n_{(1)}m_{(1)}
\end{align*}
for all $a\in A,\ m\in M,\ n\in N$. 
Then, the Yetter-Drinfel'd $A$-modules and the Yetter-Drinfel'd homomorphisms make a monoidal category together with the above tensor products.
This category is called the Yetter-Drinfel'd category, and denoted by ${}_A\mathcal{YD}^A$. 
By Yetter \cite{Yetter} it is proved that 
the Yetter-Drinfel'd category  ${}_A\mathcal{YD}^A$ has a prebraiding $c$, which is a collection of Yetter-Drinfel'd homomorphisms 
$c_{M,N} : M\otimes N\longrightarrow N\otimes M$  defined by 
$$c_{M,N}(m\otimes n)=\sum n_{(0)}\otimes (n_{(1)}\cdot m)$$
for all $m\in M,\ n\in N$. 
Furthermore, he also proved that if $A$ is a Hopf algebra, then the prebraiding $c$ is a braiding for ${}_A\mathcal{YD}^A$. 

\par 
In \cite{Radford3} Radford constructed a functor from ${}_A\Mod$ to ${}_A\mathcal{YD}^A$, that is a right adjoint of the forgetful functor $R_A: {}_A\mathcal{YD}^A\longrightarrow {}_A\Mod$ as shown in \cite{HuZhang}.  

\par \smallskip \noindent 
\begin{lem}[{\bf Radford\cite[Proposition 2]{Radford3}, Hu and Zhang\cite[Lemma 2.1]{HuZhang}}]\label{2-1}
Let $A$ be a bialgebra over $\boldsymbol{k}$, and suppose that $A^{{\rm op}}$ has antipode $\overline{S}$. 
Let $L\in {}_A\Mod$. Then 
$L\otimes A\in {}_A\mathcal{YD}^A$, where 
the left $A$ action $\cdot $ and the right $A$-coaction $\rho $ are given by 
\begin{align*}
h\cdot (l\otimes a)&=\sum (h_{(2)}\cdot l)\otimes h_{(3)}a\overline{S}(h_{(1)}),\\ 
\rho (l\otimes a)&=\sum (l\otimes a_{(1)})\otimes a_{(2)}
\end{align*}
for all $h,a\in A$ and $l\in L$. 
The correspondence $L\mapsto L\otimes A$ is extended to a functor  $I_A: {}_A\Mod\longrightarrow  {}_A\mathcal{YD}^A$ that is a right adjoint of the forgetful functor $R_A$. 
The functor $I_A$ will be referred to as Radford's induction functor derived from left $A$-modules. 
\end{lem}
 
From the above lemma there is a natural $\boldsymbol{k}$-linear isomorphism $\varphi : {\rm Hom}_{{}_A\Mod}(R_A(M), V)$ $\longrightarrow {\rm Hom}_{{}_A\mathcal{YD}^A}(M, I_A(V))$ for all $M\in {}_A\mathcal{YD}^A$ and $V\in {}_A\Mod$. 
This isomorphism is given by 
\begin{equation}\label{eq2-9}
\begin{aligned}
f\in {\rm Hom}_{{}_A\Mod}(R_A(M), V) \longmapsto  &\ \varphi (f) \in  {\rm Hom}_{{}_A\mathcal{YD}^A}(M, I_A(V)), \\ 
&\quad   \bigl( \varphi (f)\bigr) (m) =\sum f(m_{(0)})\otimes m_{(1)}\quad (m\in M). 
\end{aligned}
\end{equation}

Now we recall the following easy lemma from the category theory: 

\begin{lem}
  \label{lem:adj-transfer}
  Let $P: \mathcal{C}^{\prime} \longrightarrow \mathcal{C}$ and $Q: \mathcal{D}^{\prime} \longrightarrow \mathcal{D}$ be functors, where $\mathcal{C}$, $\mathcal{C}^{\prime}$, $\mathcal{D}$ and $\mathcal{D}^{\prime}$ are arbitrary categories, 
and suppose that there are equivalences $F: \mathcal{C} \longrightarrow \mathcal{D}$ and $F^{\prime}: \mathcal{C}^{\prime} \longrightarrow \mathcal{D}^{\prime}$ of categories such that the diagram
  \begin{equation*}
    \begin{CD}
      \mathcal{C}^{\prime} @>{F^{\prime}}>> \mathcal{D}^{\prime} \\
      @V{P}V{}V @V{}V{Q}V \\
      \mathcal{C}  @>>{F}>  \mathcal{D}
    \end{CD}
  \end{equation*}
  commutes up to isomorphism. 
If $P$ has a right (left) adjoint $I$, then $J = F^{\prime} \circ I \circ \overline{F}$ is a right (left) adjoint to $Q$, where $\overline{F}$ is a quasi-inverse functor of $F$.
\end{lem}
\begin{proof}
  We only prove the case where $P$ has a right adjoint. Let $\overline{F}^{\prime}$ be a quasi-inverse of $F^{\prime}$. 
Then there are isomorphisms
  \begin{align*}
    \mathrm{Hom}_{\mathcal{D}^{\prime}} \Big( X, J(Y) \Big)
    & \cong \mathrm{Hom}_{\mathcal{C}^{\prime}} \Big( \overline{F}^{\prime}(X), (I \circ \overline{F})(Y) \Big) \\
    & \cong \mathrm{Hom}_{\mathcal{C}}  \Big((P \circ \overline{F}^{\prime})(X),  \overline{F}(Y) \Big) \\
    & \cong \mathrm{Hom}_{\mathcal{C}}  \Big( (\overline{F} \circ Q)(X), \overline{F}(Y) \Big)
    \cong \mathrm{Hom}_{\mathcal{D}}  \Big( Q(X), Y \Big) , 
  \end{align*}
which are natural in $X \in \mathcal{D}^{\prime}$ and $Y \in \mathcal{D}$. 
Hence $J$ is a right adjoint to $Q$.
\end{proof}

If $A$ is a finite-dimensional Hopf algebra over $\boldsymbol{k}$, 
then a Yetter-Drinfel'd $A$-module $M$ becomes a left $D(A)$-module by the action given by 
$(p \bowtie a) \cdot m = \sum \langle p, (a m)_{(1)} \rangle (a m)_{(0)}$ for $p \in A^{\ast}$, $a \in A$ and $m \in M$. 
This construction establishes an isomorphism ${}_A\mathcal{YD}^A \cong {}_{D(A)}\boldsymbol{\sf M}$ of $\boldsymbol{k}$-linear braided monoidal categories (see, {\it e.g.}, \cite{Majid1}).

For a while, we denote by $F^{\prime} : {}_A\mathcal{YD}^A \longrightarrow {}_{D(A)}\boldsymbol{\sf M}$ and $R^{\prime}_A:{}_{D(A)}\boldsymbol{\sf M} \longrightarrow {}_{A}\boldsymbol{\sf M}$ the above isomorphism and the restriction functor, respectively. 
Since $R_A^{\prime} \circ F^{\prime} = R_A$ ($= {\rm id} \circ R_A$), 
the functor $I^{\prime}_A = F^{\prime}\circ I_A$ is a right adjoint to $R'_A$ by Lemma~\ref{lem:adj-transfer}. 
In what follows, based on this observation, 
we identify ${}_{D(A)} \boldsymbol{\sf M}$ with ${}_{A}\mathcal{YD}^A$ via the isomorphism $F^{\prime}$ and, abusing notation, write $R^{\prime}_A$ and $I^{\prime}_A$ as $R_A$ and $I_A$, respectively. 

\par \smallskip \noindent 
\begin{prop}\label{2-3}
Let $A$ be a finite-dimensional Hopf algebra over $\boldsymbol{k}$. 
Then, the $\boldsymbol{k}$-linear map $\Phi : {}_{D(A)}A\longrightarrow I_A(\boldsymbol{k})$ defined by $\Phi (a)=1\otimes S^{-1}(a)$ for all $a\in A$ is an isomorphism of left $D(A)$-modules. 
Here, $\boldsymbol{k}$ in $I_A(\boldsymbol{k})$ means the trivial left $A$-module. 
\end{prop}
\begin{proof}
Let $V$ be a left $A$-module. 
The left $A^{\ast}$ action corresponding to the right coaction of $I_A(V)\in {}_A\mathcal{YD}^A$ is given by 
$$p\cdot (v\otimes a)
=\sum \langle p, (v\otimes a)_{(1)}\rangle (v\otimes a)_{(0)}
=\sum \langle p, a_{(2)}\rangle \kern0.2em v\otimes a_{(1)}$$
for all $p\in A^{\ast},\ v\in V,\ a\in A$. 
Thus, the left $D(A)$-action on $I_A(V)$ is given as follows. 
\begin{align*}
(p\bowtie h)\cdot (v\otimes a)
&=p\cdot (h\cdot (v\otimes a)) \\ 
&=\sum p\cdot ((h_{(2)}\cdot v)\otimes h_{(3)}aS^{-1}(h_{(1)})) \\ 
&=\sum \langle p,\ (h_{(3)}aS^{-1}(h_{(1)}))_{(2)}\rangle \kern0.2em (h_{(2)}\cdot v)\otimes (h_{(3)}aS^{-1}(h_{(1)}))_{(1)} \\ 
&=\sum \langle p,\ h_{(5)}
a_{(2)}S^{-1}(h_{(1)})\rangle \kern0.2em (h_{(3)}\cdot v)\otimes (h_{(4)}a_{(1)}S^{-1}(h_{(2)})) 
\end{align*}
for all $p\in A^{\ast},\ h,a\in A,\ v\in V$. 
In particular, when $V$ is the trivial $A$-module $\boldsymbol{k}$, 
\begin{align*}
(p\bowtie  h)\cdot (1\otimes a)
&=\sum \langle p,\ h_{(4)}
a_{(2)}S^{-1}(h_{(1)})\rangle \kern0.2em 1\otimes (h_{(3)}a_{(1)}S^{-1}(h_{(2)})) \\ 
&=\sum \langle S^{-1}(p),\ h_{(1)}
S(a_{(2)})S(h_{(4)})\rangle \kern0.2em 1\otimes S^{-1}(h_{(2)}S(a_{(1)})S(h_{(3)})). 
\end{align*}
So, identifying $\boldsymbol{k}\otimes H=H$, we have 
$$
(p\bowtie  h)\cdot a=
\sum \langle S^{-1}(p),\ h_{(1)}
S(a_{(2)})S(h_{(4)})\rangle \kern0.2em  S^{-1}(h_{(2)}S(a_{(1)})S(h_{(3)})).$$
This is equivalent to 
\begin{equation}\label{eq2-8}
S((p\bowtie  h)\cdot  S^{-1}(a))=\sum \langle S^{-1}(p),\ h_{(1)}
a_{(1)}S(h_{(4)})\rangle \kern0.2em  h_{(2)}a_{(2)}S(h_{(3)}). 
\end{equation}
The right-hand side of the above equation coincides with the left action of the Schr\"{o}dinger module  ${}_{D(A)}A$. 
\end{proof} 

\par \medskip 
As in a similar to Lemma~\ref{2-1}, the following holds. 

\par \smallskip \noindent 
\begin{lem}[{\bf Radford\cite[Proposition 1]{Radford3}, Hu and Zhang\cite[Remark 2.2]{HuZhang}}]\label{2-4}
Let $A$ be a bialgebra over $\boldsymbol{k}$, and suppose that $A^{{\rm op}}$ has antipode $\overline{S}$. 
Let $N\in \Mod^A$. Then 
$A\otimes N\in {}_A\mathcal{YD}^A$, where 
the left $A$ action $\cdot $ and the right $A$-coaction $\rho $ are given by 
\begin{align*}
h\cdot (a\otimes n)&=ha\otimes n,\\ 
\rho (h\otimes n)&=\sum (h_{(2)}\otimes n_{(0)})\otimes h_{(3)}n_{(1)}\overline{S}(h_{(1)})
\end{align*}
for all $h,a\in A$ and $n\in N$. 
The correspondence $N \mapsto A\otimes N$ is extended to a functor 
$I^A: \Mod^A \longrightarrow {}_A\mathcal{YD}^A$ that  is a left adjoint of the forgetful functor $R^A$.  
The functor $I^A$ will be referred to as Radford's induction functor derived from right $A$-comodules. 
\end{lem} 

\par \smallskip 
Let $A$ be a finite-dimensional Hopf algebra over $\boldsymbol{k}$. 
The Hopf algebra embedding $\jmath : A^{\ast{\rm cop}}\hookrightarrow D(A)$ defined by $\jmath (p)=p\bowtie  1_A$ induces a monoidal functor ${}_{D(A)}\Mod \longrightarrow {}_{A^{\ast{\rm cop}}}\Mod\cong (\Mod^A)^{{\rm rev}}$. 
Here, $(\Mod^A)^{{\rm rev}}$ means the reversed monoidal category of $\Mod^A$. 
Under the identification ${}_A\mathcal{YD}^A={}_{D(A)}\Mod$, 
the above functor ${}_{D(A)}\Mod \longrightarrow (\Mod^A)^{{\rm rev}}$ coincides with the forgetful functor $R^A$ given in Lemma~\ref{2-4}. 
So, we denote this functor by $R^A$, again. 
By Lemmas~\ref{lem:adj-transfer} and \ref{2-4}, we see that  
the  Radford's induction  functor 
$I^A: (\Mod^A)^{{\rm rev}} \longrightarrow {}_A\mathcal{YD}^A= 
{}_{D(A)}\Mod$ is a left adjoint of $R^A: {}_{D(A)}\Mod\longrightarrow (\Mod^A)^{{\rm rev}}$. 

\par \smallskip 
Let $A$ be a finite-dimensional Hopf algebra over $\boldsymbol{k}$. 
Then the bialgebra $A^{{\rm op}}$ has an antipode since the antipode $S$ of $A$ is bijective \cite[Theorem 1]{Ra0}. 
Let $M=(M, \cdot , \rho _M)$ be a finite-dimensional Yetter-Drinfel'd $A$-module. 
Then the dual vector space $M^{\ast}$ becomes a Yetter-Drinfel'd $A$-module with respect to the following action $\cdot $ and coaction $\rho _{M^{\ast}}$: 
\begin{align*}
(a\cdot \alpha )(m)&=\alpha (S(a)\cdot m)\qquad (a\in A, \ \alpha \in M^{\ast},\ m\in M),\\ 
\rho _{M^{\ast}}(e_j^{\ast})&=\sum\limits_{i=1}^d e_i^{\ast}\otimes S^{-1}(a_{ji})\qquad (j=1,\dots ,d), 
\end{align*}

\noindent 
where $\{ e_i\}_{i=1}^d$ is a basis for $M$ and $\{ e_i^{\ast}\}_{i=1}^d$ is its dual basis, 
and $\rho _M(e_j)=\sum_{i=1}^d e_i\otimes a_{ij}$ for $j=1,\dots , d$. 

\par \smallskip \noindent 
\begin{prop}\label{2-5}
Let $A$ be a finite-dimensional Hopf algebra over $\boldsymbol{k}$. 
Then, the $\boldsymbol{k}$-linear map
 $\Phi : {}_{D(A)}A^{\ast{\rm cop}}\longrightarrow \bigl( I^A(\boldsymbol{k})\bigr)^{\ast}$ 
 defined by $\Phi (q)=S^{-1}(q)\otimes 1$ for all $q\in A^{\ast}$ is an isomorphism of left $D(A)$-modules. 
Here, $\boldsymbol{k}$ in $I^A(\boldsymbol{k})$ means the trivial right $A$-comodule. 
\end{prop}
\begin{proof}
Let $W$ be a finite-dimensional right $A$-comodule. 
Applying Radford's induction functor to  the dual $A$-comodule $W^{\ast}$, 
we have a Yetter-Drinfel'd $A$-module $I^A(W^{\ast})$ $=A\otimes W^{\ast}$. 
As mentioned before Proposition~\ref{2-5}, 
the dual $\bigl( I^A(W^{\ast})\bigr)^{\ast}$ is also a Yetter-Drinfel'd $A$-module whose left action is given by 
\begin{equation}
(h\cdot f)(a\otimes \alpha ) 
=f(S(h)\cdot (a \otimes \alpha )) 
=f(S(h)a\otimes \alpha )
\end{equation}
for all $f \in (A\otimes W^{\ast})^{\ast},\ h, a \in H,\ \alpha \in W^{\ast}$. 
The coaction $\psi$ of $\bigl( I^A(W^{\ast})\bigr)^{\ast}$ is given as follows. 
Let $\{ x_s\}_{j=1}^r$ and $\{ e_i\}_{i=1}^d$ be bases for $A$ and $W$, respectively. 
Let $\{ f_{si}\}_{\substack{1\leq s\leq r \\ 1\leq i\leq d}}$ be the dual basis for $\{ x_s\otimes e_i^{\ast}\} _{\substack{1\leq s\leq r \\ 1\leq i\leq d}}$, that is, 
$f_{si}\in  (A\otimes W^{\ast})^{\ast}$ and 
$f_{si}(x_t\otimes e_j^{\ast})=\delta _{s,t}\delta _{i,j}$ for all $i,j=1,\dots , d$ and $s, t=1,\dots , r$. 
Then, 
$$\psi (f_{si})
=\sum\limits_{t=1}^r\sum\limits _{j=1}^d f_{tj}\otimes S^{-1}(a_{(s,i), (t,j)}),$$
where $a_{(t,j), (s,i)}\in A$ is defined by 
$\rho (x_s\otimes e_i^{\ast})=\sum_{j=1}^d\sum_{t=1}^r (x_t\otimes e_j^{\ast})\otimes a_{(t,j), (s,i)}$. 
Let $\rho_W$ be the coaction of $W$, and write $\rho_W(e_j)=\sum _{i=1}^d e_i\otimes a_{ij}\ (a_{ij}\in A)$. Since 
\begin{align*}
\rho (x_s\otimes e_i^{\ast})
&=\sum ((x_s)_{(2)}\otimes e_j^{\ast})\otimes (x_s)_{(3)}S^{-1}(a_{ij})S^{-1}((x_s)_{(1)}) \\ 
&=\sum \langle (x_s)_{(2)}, x_t^{\ast}\rangle (x_t\otimes e_j^{\ast})\otimes (x_s)_{(3)}S^{-1}((x_s)_{(1)}a_{ij})
\end{align*}
by Lemma~\ref{2-4}(1), we see that 
$a_{(t,j), (s,i)}=\sum \langle (x_s)_{(2)}, x_t^{\ast}\rangle (x_s)_{(3)}S^{-1}((x_s)_{(1)}a_{ij})$, 
and whence 
\begin{align*}
\psi (f_{si})
&=\sum\limits_{t=1}^r\sum\limits_{j=1}^d \langle (x_t)_{(2)}, x_s^{\ast}\rangle \kern0.2em 
f_{tj}\otimes S^{-1}\bigl( (x_t)_{(3)}S^{-1}((x_t)_{(1)}a_{ji})\bigr) .
\end{align*}

Under the canonical identification $(A\otimes W^{\ast})^{\ast}\cong W^{\ast\ast }\otimes A^{\ast}\cong W\otimes A^{\ast}$ as vector spaces, 
we regard $W\otimes A^{\ast}$ as a Yetter-Drinfel'd $A$-module. 
Then, from the above observation it follows that the action and coaction on $W\otimes A^{\ast}$ are given by 
\begin{align*}
a\cdot (w\otimes p)
&=w\otimes (a\cdot p), \\ 
\psi (w\otimes p)
&=\sum \limits_{t=1}^r \langle p, (x_t)_{(2)}\rangle (w_{(0)}\otimes x_t^{\ast})\otimes  S^{-1}\bigl( (x_t)_{(3)}S^{-1}((x_t)_{(1)}w_{(1)})\bigr)
\end{align*}
for all $a\in A, w\in W,\ p\in A^{\ast}$, where 
$a\cdot p$ is the action of $A$ on $A^{\ast}$ as the dual $A$-module of the left regular $A$-module. 
\par 
Considering the case where $W$ is the trivial right $A$-comodule $\boldsymbol{k}$, 
we have a Yetter-Drinfel'd $A$-module structure on $A^{\ast}$ whose action and coaction are given by 
\begin{align*}
a\cdot q
&=\sum\limits_{t=1}^r \langle a\cdot q, x_t\rangle x_t^{\ast} 
=\sum\limits_{t=1}^r \langle q, S(a)x_t\rangle x_t^{\ast} 
=\sum \langle q_{(1)}, S(a)\rangle q_{(2)}, \\ 
\psi (p)&=\sum \limits_{t=1}^r \langle p, (x_t)_{(2)}\rangle x_t^{\ast}\otimes  S^{-1}\bigl( (x_t)_{(3)}S^{-1}((x_t)_{(1)})\bigr).
\end{align*}

So, the corresponding left $D(A)$-module structure on $A^{\ast}$ is given by 
\begin{align*}
(p\bowtie h)\cdot q
&=p\cdot (h\cdot q) \\ 
&=\sum \sum \limits_{t=1}^r \langle h\cdot q, (x_t)_{(2)}\rangle \langle p, S^{-1}\bigl( (x_t)_{(3)}S^{-1}((x_t)_{(1)})\bigr) \rangle 
x_t^{\ast} \displaybreak[0]\\ 
&=\sum \sum \limits_{t=1}^r \langle h\cdot q, (x_t)_{(2)}\rangle \langle S^{-1}(p_{(2)}),  (x_t)_{(3)}\rangle \langle 
S^{-2}(p_{(1)}),  (x_t)_{(1)}\rangle x_t^{\ast} \\ 
&=\sum \sum \limits_{t=1}^r \langle S^{-2}(p_{(1)})(h\cdot q)S^{-1}(p_{(2)}) , x_t\rangle x_t^{\ast} \displaybreak[0]\\ 
&=\sum \langle q_{(1)}, S(h)\rangle S^{-2}(p_{(1)}) q_{(2)}S^{-1}(p_{(2)})
\end{align*}
for all $p,q\in A^{\ast}$ and $a\in A$. 
Therefore, the left $D(A)$ action on $A^{\ast}\cong \bigl(I^A(\boldsymbol{k}^{\ast})\bigr) ^{\ast}$ satisfies the equation 
\begin{equation}\label{eq2-10}
S((p\bowtie h)\cdot S^{-1}(q))=\sum \langle q_{(2)}, h\rangle p_{(2)} q_{(1)}
S^{-1}(p_{(1)}).
\end{equation}
The right-hand side coincides with the left $D(A)$ action on the dual Schr\"{o}dinger module ${}_{D(A)}A^{\ast{\rm cop}}$. 
Since $\boldsymbol{k}^{\ast}\cong \boldsymbol{k}$ as right $A$-comodules, the equation (\ref{eq2-10}) gives rise to an isomorphism 
$\bigl(I^A(\boldsymbol{k})\bigr) ^{\ast}\cong {}_{D(A)}A^{\ast{\rm cop}}$. 
\end{proof}

\subsection{Tensor products of Schr\"{o}dinger modules}

In this subsection, we compute the tensor product of Schr\"{o}dinger modules by using Propositions~\ref{2-3} and~\ref{2-5}. 
First, we provide the following lemma, which is proved straightforwardly.

\begin{lem}\label{2-10}
  Let $A$ be a Hopf algebra over $\boldsymbol{k}$ with bijective antipode. Then:
  \begin{itemize}
  \item [(1)] There is a natural isomorphism of Yetter-Drinfel'd modules
    \begin{equation*}
      \Phi: I_A(V) \otimes M \longrightarrow I_A(V \otimes R_A(M))
      \quad (V \in {}_A\Mod, M \in {}_A\mathcal{YD}^A)
    \end{equation*}
    given by $\Phi(v \otimes a \otimes m) = \sum v \otimes m_{(0)} \otimes m_{(1)} a$ for $v \in V$, $a \in A$ and $m \in M$.
  \item [(2)] There is a natural isomorphism of Yetter-Drinfel'd modules
    \begin{equation*}
      \Psi: I^A(V \otimes R^A(M)) \longrightarrow I^A(V) \otimes M
      \quad (V \in \Mod^A, M \in {}_A\mathcal{YD}^A)
    \end{equation*}
    given by $\Psi(a \otimes v \otimes m) = \sum a_{(1)} \otimes v \otimes a_{(2)} m$ for $v \in V$, $a \in A$ and $m \in M$.
  \end{itemize}
\end{lem}

For a Hopf algebra $A$, we denote by $A_{\mathrm{ad}}$ the adjoint representation of $A$, \textit{i.e.}, the vector space $A$ endowed with the left $A$-module structure given by~\eqref{eq2-1}.

\begin{prop}
  \label{prop:Sch-mod-tensor}
  Let $A$ be a finite-dimensional Hopf algebra over $\boldsymbol{k}$. Then
  \begin{equation*}
    ({}_{D(A)}A)^{\otimes n} \cong I_A(A_{\mathrm{ad}}^{\otimes (n-1)}).
  \end{equation*}
\end{prop}
\begin{proof}
By Proposition~\ref{2-3} and Lemma~\ref{2-10},
  \begin{equation*}
    ({}_{D(A)}A)^{\otimes n}
    \cong I_A(\boldsymbol{k}) \otimes ({}_{D(A)} A)^{\otimes (n - 1)}
    \cong (I_A \circ R_A)(({}_{D(A)} A)^{\otimes (n - 1)})
    \cong I_A(A_{\mathrm{ad}}^{\otimes (n-1)}). \qedhere
  \end{equation*}
\end{proof}

The following result is a non-semisimple generalization of a part of \cite[Proposition 4]{Burciu}.

\begin{prop}
  $({}_{D(A)}A) \otimes ({}_{D(A)}A^{\ast{\rm cop}}) \cong D(A)$ as left $D(A)$-modules.
\end{prop}
\begin{proof}
  Let, in general, $H$ be a finite-dimensional Hopf algebra. As is well-known, there are natural isomorphisms of vector spaces
  \begin{equation}
    \label{eq:prop-2-11-proof-1}
    {\rm Hom}_H(X \otimes H, Y) \cong {\rm Hom}_{\boldsymbol{k}}(X, Y) \cong {\rm Hom}_H(X, Y \otimes H)
    \quad (X, Y \in {}_H \boldsymbol{\sf M}).
  \end{equation}
  Since $(R_A\circ I^A)(\boldsymbol{k}) \cong A$ as left $A$-modules, there are natural isomorphisms of vector spaces
  \begin{align*}
    {\rm Hom}_{D(A)}(X, ({}_{D(A)}A) \otimes ({}_{D(A)}A^{\ast{\rm cop}}))
    & \cong {\rm Hom}_{D(A)}(X, I_A(\boldsymbol{k}) \otimes I^A(\boldsymbol{k})^*) \\
    & \cong {\rm Hom}_{D(A)}(X \otimes I^A(\boldsymbol{k}), I_A(\boldsymbol{k})) \\
    & \cong {\rm Hom}_{A}(R_A(X \otimes I^A(\boldsymbol{k})), \boldsymbol{k}) \\
    & \cong {\rm Hom}_{A}(R_A(X) \otimes A, \boldsymbol{k}) \\
    & \cong R_A(X)^{\ast} = X^{\ast}
    \quad \text{(by~(\ref{eq:prop-2-11-proof-1}))}
  \end{align*}
  for $X \in {}_{D(A)} \boldsymbol{\sf M}$. On the other hand, ${\rm Hom}_{D(A)}(X, D(A)) \cong X^{\ast}$ by~(\ref{eq:prop-2-11-proof-1}). Hence the result follows from Yoneda's lemma.
\end{proof}

\section{Categorical aspects of Schr\"{o}dinger modules}

\subsection{The center of a monoidal category}

We first recall the center construction for monoidal categories, which was introduced by Drinfel'd (see Joyal and Street \cite{JS2} and Majid \cite{Majid2}).
Let $\mathcal{C}=(\mathcal{C}, \otimes, \mathbb{I}, a, r, l)$ be a monoidal category, 
and let $V \in \mathcal{C}$ be an object. A \textit{half-braiding} for $V$ is a natural isomorphism $c_{-,V}: (-) \otimes V \longrightarrow V \otimes (-)$ such that the diagram
\begin{equation*}
  \begin{CD}
    (X \otimes Y) \otimes V
    @>{c_{X \otimes Y, V}}>{}>
    V \otimes (X \otimes Y)
    @<{a_{V,X,Y}^{-1}}<{}<
    (V \otimes X) \otimes Y \\
    @V{a_{X,Y,V}}V{}V @.
    @A{}A{c_{X,V} \otimes \mathrm{id}_Y}A \\
    X \otimes (Y \otimes V)
    @>>{\mathrm{id}_X \otimes c_{Y,V}}>
    X \otimes (V \otimes Y)
    @>>{a_{X,V,Y}^{-1}}>
    (X \otimes V) \otimes Y
  \end{CD}
\end{equation*}
commutes for all $X, Y \in \mathcal{C}$. 
The {\em center} of $\mathcal{C}$, denoted by $\mathcal{Z}(\mathcal{C})$, is the category defined as follows: 
An object of $\mathcal{Z}(\mathcal{C})$ is a pair $(V, c_{-,V})$ consisting of an object $V \in \mathcal{C}$ and a half braiding $c_{-,V}$ for $V$. 
If $(V, c_{-,V})$ and $(W, c_{-,W})$ are objects of $\mathcal{Z}(\mathcal{C})$, then a \textit{morphism} from $(V, c_{-,V})$ to $(W, c_{-,W})$ is a morphism $f: V \longrightarrow W$ in $\mathcal{C}$ such that
\begin{equation*}
  (f \otimes \mathrm{id}_X) \circ c_{X,V} = c_{X,W} \circ (\mathrm{id}_Y \otimes f)
\end{equation*}
for all $X \in \mathcal{C}$. 
The composition of morphisms in $\mathcal{Z}(\mathcal{C})$ is given in an obvious way. 
Forgetting the half-braiding defines a faithful functor $\Pi_{\mathcal{C}}: \mathcal{Z}(\mathcal{C}) \longrightarrow \mathcal{C}$, which will be referred to as the \textit{forgetful functor} from $\mathcal{Z}(\mathcal{C})$ to $\mathcal{C}$.

The category $\mathcal{Z}(\mathcal{C})$ is in fact a braided monoidal category: 
First, the tensor product of objects of $\mathcal{Z}(\mathcal{C})$ is defined by $  (V, c_{-,V}) \otimes (W, c_{-,W}) = (V \otimes W, c_{-,V \otimes W})$, where the half-braiding $c_{-,V \otimes W}$ for $V \otimes W$ is a unique natural isomorphism making the diagram
\begin{equation*}
  \begin{CD}
    X \otimes (V \otimes W)
    @>{c_{X, V \otimes W}}>{}>
    (V \otimes W) \otimes X
    @<{a_{X,V,W}}<{}<
    (X \otimes V) \otimes W \\
    @V{a_{X,Y,V}}V{}V @.
    @A{}A{\mathrm{id}_W \otimes c_{X,W}}A \\
    (X \otimes V) \otimes W
    @>>{c_{X,V} \otimes \mathrm{id}_X}>
    (V \otimes X) \otimes W
    @>>{a_{V,X,W}}>
    V \otimes (X \otimes W)
  \end{CD}
\end{equation*}
commute for all $X \in \mathcal{C}$. 
The unit object is $\mathbb{I}_{\mathcal{Z}(\mathcal{C})} := (\mathbb{I}, l^{-1} \circ r)$. 
The associativity and the unit constraints for $\mathcal{Z}(\mathcal{C})$ are defined so that the forgetful functor $\Pi_{\mathcal{C}}$ is strict monoidal. Finally, the braiding of $\mathcal{Z}(\mathcal{C})$ is given by
\begin{equation*}
  c = \Big\{
  c_{V,W}: (V, c_{-,V}) \otimes (W, c_{-,W}) \longrightarrow 
  (W, c_{-,W}) \otimes (V, c_{-,V})
  \Big\}
  _{(V, c_{-,V}), (W, c_{-,W})\in \mathcal{Z}(\mathcal{C})}.
\end{equation*}

Since the center construction is described purely in terms of monoidal categories, it is natural to expect that equivalent monoidal categories have equivalent centers. 
We omit to give a proof of the following well-known fact, since the proof is easy but quite long.

\begin{lem}
  \label{lem:Z-functoriality}
  For each monoidal equivalence $F: \mathcal{C} \longrightarrow \mathcal{D}$ between monoidal categories $\mathcal{C}$ and $\mathcal{D}$, there exists a unique braided monoidal equivalence
  \begin{equation*}
    \mathcal{Z}(F): \mathcal{Z}(\mathcal{C}) \longrightarrow \mathcal{Z}(\mathcal{D})
  \end{equation*}
  such that $\Pi_{\mathcal{D}} \circ \mathcal{Z}(F) = \Pi_{\mathcal{C}} \circ F$ as monoidal functors. The assignment $F \mapsto \mathcal{Z}(F)$ enjoys the following properties:
  \begin{itemize}
  \item[(1)] If $\displaystyle \mathcal{C} \mathop{\longrightarrow}^{F} \mathcal{D} \mathop{\longrightarrow}^{G} \mathcal{E}$ is a sequence of monoidal equivalences between monoidal categories $\mathcal{C}$, $\mathcal{D}$ and $\mathcal{E}$, then $\mathcal{Z}(G \circ F) = \mathcal{Z}(G) \circ \mathcal{Z}(F)$.
  \item[(2)] Let $F_1, F_2: \mathcal{C} \longrightarrow \mathcal{D}$ be monoidal equivalences. For each monoidal natural transformation $\alpha: F_1 \longrightarrow F_2$, a monoidal natural transformation
    \begin{equation*}
      \mathcal{Z}(\alpha): \mathcal{Z}(F_1) \longrightarrow \mathcal{Z}(F_2)
    \end{equation*}
    is defined by $\mathcal{Z}(\alpha)_{\mathbf{X}} = \alpha_{X}$ for $\mathbf{X} = (X, \sigma_X) \in \mathcal{Z}(\mathcal{C})$. Moreover, $\alpha \mapsto \mathcal{Z}(\alpha)$ preserves the horizontal and the vertical compositions of natural transformations.
  \end{itemize}
\end{lem}

By a \textit{$\boldsymbol{k}$-linear monoidal category}, 
we mean a monoidal category $\mathcal{C}$ such that $\mathrm{Hom}_{\mathcal{C}}(X, Y)$ is a vector space over $\boldsymbol{k}$ for all objects $X, Y \in \mathcal{C}$, and the composition and the tensor product of morphisms in $\mathcal{C}$ are linear in each variables. 
We note that if $\mathcal{C}$ is a $\boldsymbol{k}$-linear monoidal category, then so is $\mathcal{Z}(\mathcal{C})$ in such a way that the functor $\Pi_{\mathcal{C}}$ is $\boldsymbol{k}$-linear. 
If $F: \mathcal{C} \longrightarrow \mathcal{D}$ is a $\boldsymbol{k}$-linear monoidal equivalence between $\boldsymbol{k}$-linear monoidal categories, then $\mathcal{Z}(F)$ is $\boldsymbol{k}$-linear.

\subsection{Schr\"{o}dinger modules in monoidal categories}

Let $\mathcal{C}$ be a monoidal category. 
We say that an object $\mathbf{A} \in \mathcal{Z}(\mathcal{C})$ is a {\em Schr\"{o}dinger object} if there exists a bijection
$\mathrm{Hom}_{\mathcal{C}}(\Pi_{\mathcal{C}}(\mathbf{X}), \mathbb{I}_{\mathcal{Z}(\mathcal{C})})
\cong \mathrm{Hom}_{\mathcal{Z}(\mathcal{C})}(\mathbf{X}, \mathbf{A})$, 
which is natural in the variable $\mathbf{X} \in \mathcal{Z}(\mathcal{C})$. 
Note that such an object is unique up to isomorphism by Yoneda's lemma (if it exists).

\begin{lem}
  \label{lem:Sch-obj-invariance}
  Let $\mathcal{C}$ and $\mathcal{D}$ be monoidal categories.
  \begin{itemize}
  \item[(1)] Suppose that $\Pi_{\mathcal{C}}: \mathcal{Z}(\mathcal{C}) \longrightarrow \mathcal{C}$ has a right adjoint functor $I_{\mathcal{C}}: \mathcal{C} \longrightarrow \mathcal{Z}(\mathcal{C})$. 
Then the object $I_{\mathcal{C}}(\mathbb{I})$ is a Schr\"{o}dinger object for $\mathcal{C}$.
  \item[(2)] Suppose that $\mathbf{A}_{\mathcal{C}}$ is a Schr\"{o}dinger object for $\mathcal{C}$, 
and there exists an equivalence $F: \mathcal{C} \longrightarrow \mathcal{D}$ of monoidal categories. 
Then $\mathcal{Z}(F) (\mathbf{A}_{\mathcal{C}})$ is a Schr\"{o}dinger object for $\mathcal{D}$.
  \end{itemize}
\end{lem}
\begin{proof}
  Part (1) follows immediately from the definition of the adjoint functor. 
Part (2) is shown as follows: Let $\overline{F}$ be a quasi-inverse of $F$. 
Then $\mathcal{Z}(\overline{F})$ is a quasi-inverse of $\mathcal{Z}(F)$. Hence,
  \begin{gather*}
    \mathrm{Hom}_{\mathcal{Z}(\mathcal{D})}
    \big( \mathbf{X}, \mathcal{Z}(F)(\mathbf{A}_{\mathcal{C}}) \big)
    \cong \mathrm{Hom}_{\mathcal{Z}(\mathcal{C})}
    \big( \mathcal{Z}(\overline{F})(\mathbf{X}), \mathbf{A}_{\mathcal{C}} \big)
    \cong \mathrm{Hom}_{\mathcal{C}}
    \big( (\Pi_{\mathcal{C}} \circ \mathcal{Z}(F))(\mathbf{X}), \mathbb{I} \big) \\
    \cong \mathrm{Hom}_{\mathcal{D}}
    \big( (F \circ \Pi_{\mathcal{D}})(\mathbf{X}), \mathbb{I} \big)
    \cong \mathrm{Hom}_{\mathcal{D}}
    \big( \Pi_{\mathcal{D}}(\mathbf{X}), \overline{F}(\mathbb{I}) \big)
    \cong \mathrm{Hom}_{\mathcal{D}}
    \big( \Pi_{\mathcal{D}}(\mathbf{X}), \mathbb{I} \big)
  \end{gather*}
  for all $\mathbf{X} \in \mathcal{Z}(\mathcal{C})$. 
Therefore, $\mathcal{Z}(F) (\mathbf{A}_{\mathcal{C}})$ is the Schr\"{o}dinger object for $\mathcal{D}$.
\end{proof}

Now, let $A$ be a Hopf algebra over $\boldsymbol{k}$ with bijective antipode. Then
\begin{equation}
  \label{eq:cat-iso-Z-YD}
  \Phi_A: \mathcal{Z}({}_A\boldsymbol{\sf M}) \longrightarrow {}_A\mathcal{YD}^A,
  \quad (V, c_{-,V}) \mapsto (V, \rho),
\end{equation}
where $\rho: V \to V \otimes A$ is the linear map given by $\rho(v) = c_{A,V}(1_A \otimes v)$ ($v \in V$),
is an isomorphism of $\boldsymbol{k}$-linear braided monoidal categories \cite{JS2, Majid2}.

\begin{thm}
  Let $A$ and $B$ be Hopf algebras over $\boldsymbol{k}$ with bijective antipode. 
Suppose that there is an equivalence $F: {}_A\Mod \longrightarrow {}_B\Mod$ of $\boldsymbol{k}$-linear monoidal categories.
  \begin{itemize}
  \item[(1)] $\mathbf{S}_A := (\Phi_A^{-1} \circ I_A)(\boldsymbol{k}) \in \mathcal{Z}({}_A\Mod)$ is a Schr\"{o}dinger object for ${}_A\Mod$.
  \item[(2)] There uniquely exists an equivalence
    $\widetilde{F}: {}_A \mathcal{YD}^A \longrightarrow {}_B \mathcal{YD}^B$
    of $\boldsymbol{k}$-linear braided monoidal categories such that $R_B \circ \widetilde{F} = F \circ R_A$ as monoidal functors.
  \item[(3)] The functor $\widetilde{F}$ in Part (2) satisfies $I_B \circ F \cong \widetilde{F} \circ I_A$.
  \end{itemize}
\end{thm}
\begin{proof}
  To prove Part (1), note that the diagram
  \begin{equation}
    \label{eq:cat-iso-Z-YD-res}
    \begin{CD}
      \mathcal{Z}({}_A\Mod) @>{\Phi_A}>> {}_A\mathcal{YD}^A \\
      @V{\Pi_{{}_A\Mod}}V{}V @V{}V{R_A}V \\
      {}_A\Mod @>>{\mathrm{id}}> {}_A\Mod
    \end{CD}
  \end{equation}
  commutes. 
Applying Lemma~\ref{lem:adj-transfer} to this diagram, 
we see that $\Phi_A^{-1} \circ I_A$ is a right adjoint functor to $\Pi_{{}_A\Mod}$. 
Hence $\mathbf{S}_A$ is a Schr\"{o}dinger object for ${}_A\Mod$ by Part (1) of Lemma~\ref{lem:Sch-obj-invariance}.

  Now we show Part (2). 
It is easy to see that $\widetilde{F} = \Phi_B^{-1} \circ \mathcal{Z}(F) \circ \Phi_A$ satisfies the required conditions. 
To show the uniqueness, note that~\eqref{eq:cat-iso-Z-YD-res} is in fact a commutative diagram of monoidal functors. 
Thus, if $\widetilde{F}$ satisfies the required conditions, then, by \eqref{eq:cat-iso-Z-YD-res},
  \begin{equation*}
    \Pi_{{}_B\Mod} \circ (\Phi_B \circ \widetilde{F} \circ \Phi_A^{-1})
    = R_B \circ \widetilde{F} \circ \Phi_A^{-1}
    = F \circ R_A \circ \Phi_A^{-1}
    = F \circ \Pi_{{}_A\Mod}
  \end{equation*}
  as monoidal functors. 
Hence $\widetilde{F} = \Phi_B^{-1} \circ \mathcal{Z}(F) \circ \Phi_A$ by the uniqueness part of Lemma \ref{lem:Z-functoriality}.

  To prove Part (3), we recall that $I_A$ is a right adjoint to $R_A$. 
Let $\overline{F}$ be a quasi-inverse of $F$. 
By Part (2) and Lemma \ref{lem:adj-transfer}, $I' := \widetilde{F} \circ I_A \circ \overline{F}$ is a right adjoint to $R_A$. 
By the uniqueness of the right adjoint functor, we have $I' \cong I_B$. 
Hence, $I_B \circ F \cong \widetilde{F} \circ I_A$.
\end{proof}

Suppose that $A$ is finite-dimensional. Then we obtain an isomorphism
\begin{equation}
  \label{eq:cat-iso-Z-Double}
  \Psi_A: \mathcal{Z}({}_A\Mod) \mathop{\longrightarrow}^{\cong} {}_{D(A)}\Mod
\end{equation}
of $\boldsymbol{k}$-linear braided monoidal categories by composing~\eqref{eq:cat-iso-Z-YD} and the isomorphism ${}_{A}\mathcal{YD}^A \cong {}_{D(A)}\Mod$ (which we have recalled in Section 2). Note that
\begin{equation*}
  \Pi_{{}_A\Mod} = R_A \circ \Psi_A
\end{equation*}
as monoidal functors, 
where $R_A$ is the restriction functor ${}_{D(A)}\Mod \longrightarrow {}_{A}\Mod$ ({\it cf}. \eqref{eq:cat-iso-Z-YD} and~\eqref{eq:cat-iso-Z-YD-res}). 
Let $I_A$ be the right adjoint functor of $R_A$ given in Section 2. 
Recall from Proposition~\ref{2-5} that the Schr\"{o}dinger module ${}_{D(A)}A$ is isomorphic to $I_A(\boldsymbol{k})$. 
By the same way as the previous theorem, we prove:

\begin{thm}
  \label{thm:Sch-mod-invariance}
  Let $A$ and $B$ be finite-dimensional Hopf algebras over the same field $\boldsymbol{k}$. 
Suppose that there is an equivalence $F: {}_A\Mod \longrightarrow {}_B\Mod$ of $\boldsymbol{k}$-linear monoidal categories.
  \begin{itemize}
  \item[(1)] $\mathbf{S}_A := \Psi_A^{-1}({}_{D(A)}A) \in \mathcal{Z}({}_A\Mod)$ is a Schr\"{o}dinger object for ${}_A\Mod$.
  \item[(2)] There uniquely exists an equivalence
    $\widetilde{F}: {}_{D(A)} \Mod \longrightarrow {}_{D(B)} \Mod$
    of $\boldsymbol{k}$-linear braided monoidal categories such that $R_B \circ \widetilde{F} = F \circ R_A$ as monoidal functors.
  \item[(3)] The functor $\widetilde{F}$ in Part (2) satisfies $I_B \circ F \cong \widetilde{F} \circ I_A$.
  \end{itemize}
\end{thm}

An important corollary is the following invariance of the Schr\"odinger module:

\begin{cor}\label{3-5}
  The equivalence $\widetilde{F}: {}_{D(A)}\Mod \longrightarrow {}_{D(B)}\Mod$ of Theorem~\ref{thm:Sch-mod-invariance} preserves the Schr\"{o}dinger module, i.e., $\widetilde{F}({}_{D(A)}A) \cong {}_{D(B)}B$.
\end{cor}

\begin{rem}
  An equivalence $G: {}_{D(A)} \Mod \longrightarrow {}_{D(B)} \Mod$ of $\boldsymbol{k}$-linear braided monoidal categories does not preserve the Schr\"{o}dinger module in general.
\end{rem}

\begin{rem}
  As Masuoka pointed out to us, Corollaries~\ref{3-5} and \ref{3-7} below can be also derived from the point of view of cocycle deformations by using the action given in \cite[Proposition 5.1]{Mas3}. 
\end{rem}

\subsection{The dual Schr\"{o}dinger module as a Schr\"{o}dinger object}

Let $A$ be a finite-dimensional Hopf algebra over $\boldsymbol{k}$. 
Recall from Section 2 that there exists an isomorphism $D(A) \cong D(A^{\ast})^{\mathrm{cop}}$ of quasitriangular Hopf algebras. 
Since $\Mod^A$ is isomorphic to ${}_{A^{\ast}}\Mod$ as $\boldsymbol{k}$-linear monoidal categories, we have isomorphisms
\begin{equation}
  \label{eq:cat-iso-Z-YD-Xi}
  \mathcal{Z}(\Mod^A)^{\mathrm{rev}}
  \cong \mathcal{Z}({}_{A^{\ast}}\Mod)^{\mathrm{rev}}
  \cong ({}_{D(A^{\ast})}\Mod)^{\mathrm{rev}}
  \cong {}_{D(A^{\ast})^{\mathrm{cop}}}\Mod
  \cong {}_{D(A)}\Mod
\end{equation}
of $\boldsymbol{k}$-linear braided monoidal categories. 
Now, let $\Xi_A: \mathcal{Z}(\Mod^A)^{\mathrm{rev}} \longrightarrow {}_{D(A)}\Mod$ be the composition of the above isomorphisms. One can check that the diagram
\begin{equation}
  \label{eq:cat-iso-Z-YD-res-comod}
  \begin{CD}
    \mathcal{Z}(\Mod^A)^{\mathrm{rev}} @>{\Xi_A}>> {}_{D(A)}\Mod \\
    @V{\Pi_{\Mod^A}}V{}V @V{}V{R^A}V \\
    (\Mod^A)^{\mathrm{rev}} @>>{\mathrm{id}}> (\Mod^A)^{\mathrm{rev}} @.
  \end{CD}
\end{equation}
commutes, where $R^A$ is the functor used in Subsection 2.1. 
By Lemma~\ref{2-4} the functor $I^A: \Mod^A \longrightarrow {}_{D(A)}\Mod$ is a left adjoint to $R^A$. 
The following theorem is proved by the same way as Theorem~\ref{thm:Sch-mod-invariance}.

\begin{thm}
  \label{thm:3-8}
  Let $A$ and $B$ be finite-dimensional Hopf algebras over the same field $\boldsymbol{k}$. 
Suppose that there is an equivalence $F: \Mod^A \longrightarrow \Mod^B$ of $\boldsymbol{k}$-linear monoidal categories.
  \begin{itemize}
  \item[(1)] There uniquely exists an equivalence
    $\widetilde{F}: {}_{D(A)} \Mod \longrightarrow {}_{D(B)} \Mod$
    of $\boldsymbol{k}$-linear braided monoidal categories such that $R^B \circ \widetilde{F} = F \circ R^A$ as monoidal functors.
  \item[(2)] The functor $\widetilde{F}$ in Part (1) satisfies $I^B \circ F \cong \widetilde{F} \circ I^A$.
  \end{itemize}
\end{thm}

In Proposition~\ref{2-5} it is shown that the dual Schr\"{o}dinger module ${}_{D(A)}A^{\ast \mathrm{cop}}$ is isomorphic to the left dual of $I^A(\boldsymbol{k})$. 
The following corollary is obtained by Part (2) of the above theorem, and the fact that a monoidal equivalence preserves left duals.

\begin{cor}\label{3-7}
  The equivalence $\widetilde{F}: {}_{D(A)}\Mod \longrightarrow {}_{D(B)}\Mod$ of Theorem~\ref{thm:3-8} preserves the dual Schr\"{o}dinger module, i.e., $\widetilde{F}({}_{D(A)}A^{\ast \mathrm{cop}}) \cong {}_{D(B)}B^{\ast \mathrm{cop}}$.
\end{cor}

We have introduced the notion of a Schr\"{o}dinger object to explain categorical nature of the Schr\"{o}dinger module. 
The following theorem claims that also the dual Schr\"{o}dinger object can be interpreted in terms of a Schr\"{o}dinger object:

\begin{thm}
  \label{thm:dual-Sch-mod-categorical-interp}
  $\mathbf{S}_A := \Xi_A^{-1}({}_{D(A)}A^{\ast \mathrm{cop}}) \in \mathcal{Z}(\boldsymbol{\sf M}^A)$ is a Schr\"{o}dinger object for $\boldsymbol{\sf M}^A$.
\end{thm}
\begin{proof}
  For simplicity, write $\mathrm{Hom}_{\boldsymbol{\sf M}^A}(X, Y)$ as ${\rm Hom}^A(X, Y)$. 
By Proposition~\ref{2-5} and Lemma~\ref{2-10}, we obtain natural isomorphisms
  \begin{align*}
    {\rm Hom}_{D(A)}(X, {}_{D(A)}A^{\ast{\rm cop}})
    & \cong {\rm Hom}_{D(A)}(I^A(\boldsymbol{k}) \otimes X, \boldsymbol{k}) \\
    & \cong {\rm Hom}_{D(A)}(I^A(\boldsymbol{k} \otimes R^A(X)), \boldsymbol{k})
    \cong {\rm Hom}^A(R^A(X), \boldsymbol{k})
  \end{align*}
  for $X \in {}_{D(A)}\boldsymbol{\sf M}$. Hence there are natural isomorphisms
  \begin{align*}
    \mathrm{Hom}_{\mathcal{Z}(\boldsymbol{\sf M}^A)}(\mathbf{X}, \mathbf{S}^A)
    & \cong \mathrm{Hom}_{D(A)}(\Xi_A(\mathbf{X}), {}_{D(A)}A^{\ast\mathrm{cop}}) \\
    & \cong \mathrm{Hom}^A \big( (R^A \circ \Xi_A)(\mathbf{X}), \boldsymbol{k} \big)
    \cong \mathrm{Hom}^A \big( \Pi_{\boldsymbol{\sf M}^A}(\mathbf{X}), \boldsymbol{k} \big)
  \end{align*}
  for $\mathbf{X} \in \mathcal{Z}(\boldsymbol{\sf M}^A)$. This shows that $\mathbf{S}_A \in \mathcal{Z}(\boldsymbol{\sf M}^A)$ is a Schr\"{o}dinger object for $\boldsymbol{\sf M}^A$.
\end{proof}

The following theorem is a slight weak version of Proposition~\ref{2-7} and its corollary. Emphasizing the role of the Schr\"odinger object, we now reexamine the proof of Proposition \ref{2-7} and its corollary.

\begin{thm}\label{3-10}
  For a finite-dimensional Hopf algebra $A$ over $\boldsymbol{k}$, we denote by
  \begin{equation}\label{eq3-7}
    F_A: ({}_{D(A^*)} \Mod)^{\mathrm{rev}} = {}_{D(A^{*})^{\mathrm{cop}}} \Mod \longrightarrow {}_{D(A)}\Mod
  \end{equation}
  the isomorphism of $\boldsymbol{k}$-linear braided monoidal categories induced by the isomorphism $\phi_A$ given by (\ref{eq2-12}). Then there are isomorphisms
  \begin{equation*}
    F_A \big( \,{}_{D(A^{\ast})}A^{\ast} \big) \cong {}_{D(A)}A^{\ast\mathrm{cop}}
    \text{\quad and \quad}
    F_A \big( \,{}_{D(A^{\ast})}(A^{\ast})^{\ast\mathrm{cop}} \big) \cong {}_{D(A)}A.
  \end{equation*}
\end{thm}
\begin{proof}
  To establish the first isomorphism, we first note that the category isomorphism $F_A$ is expressed as follows ({\it cf}. \eqref{eq:cat-iso-Z-YD-Xi}):
  \begin{equation*}
    \begin{CD}
      F_A: {}_{D(A^{\ast})^{\mathrm{cop}}} \Mod {\ \cong \ }
      @.
      ({}_{D(A^{\ast})}\Mod)^{\mathrm{rev}}
      @>{\Psi_{A^{\ast}}}>> \mathcal{Z}({}_{A^{\ast}}\Mod)^{\mathrm{rev}}
      @>{\cong}>> \mathcal{Z}(\Mod^A)^{\mathrm{rev}}
      @>{\Xi_A}>> {}_{D(A)} \Mod,
    \end{CD}
  \end{equation*}
  where the arrow $\mathcal{Z}({}_{A^{\ast}}\Mod)^{\mathrm{rev}} \longrightarrow \mathcal{Z}(\Mod^A)^{\mathrm{rev}}$ is the isomorphism induced by ${}_{A^{\ast}}\Mod \cong \Mod^A$. 
Theorem~\ref{thm:Sch-mod-invariance} tells us that ${}_{D(A^{\ast})}A^{\ast}$ is an object corresponding to a Schr\"{o}dinger object for ${}_{A^{\ast}}\Mod$. 
On the other hand, Theorem~\ref{thm:dual-Sch-mod-categorical-interp} tells us that ${}_{D(A)}A^{\ast}$ is an object corresponding to a Schr\"{o}dinger object for $\Mod^A$. 
  Hence, by Lemma~\ref{lem:Sch-obj-invariance}, we have $F_A({}_{D(A^{\ast})}A^{\ast}) \cong {}_{D(A)}A^{\ast\mathrm{cop}}$.

  Once the first isomorphism is obtained, the second isomorphism can be obtained by replacing $A$ with $A^*$ ({\it cf}. the proof of Corollary~\ref{2-8}).
\end{proof}

\par \medskip 
\section{Applications}
\par 
Motivated by the construction of quantum representations of the $n$-strand braid group $B_n$ 
due to Reshetikhin and Turaev \cite{RT}, 
a family of monoidal Morita invariants of a finite-dimensional Hopf algebra,
which is indexed by braids, can be obtained from the Schr\"{o}dinger module.

Let $A$ be a finite-dimensional Hopf algebra. 
It turns out that the invariant associated with the identity element $\boldsymbol{1} \in B_1$ is equal to the quantum dimension the Schr\"{o}dinger module ${}_{D(A)}A$ in the sense of Majid \cite{MajidBook}, 
and thus equal to $\mathrm{Tr}(S^2)$ by \cite[Example 9.3.8]{MajidBook}, 
where $S$ is the antipode of $A$ (see Bulacu-Torrecillas \cite{BulacuTorrecillas} for the case of quasi-Hopf algebras). 
As is well-known, $\mathrm{Tr}(S^2)$ has the following representation-theoretic meaning: 
$\mathrm{Tr}(S^2) \ne 0$ if and only if $A$ is semisimple and cosemisimple \cite{Rad1994}. 
In this section, we show that the invariants derived from other braids, like \raisebox{-0.4cm}{\includegraphics[height=1cm]{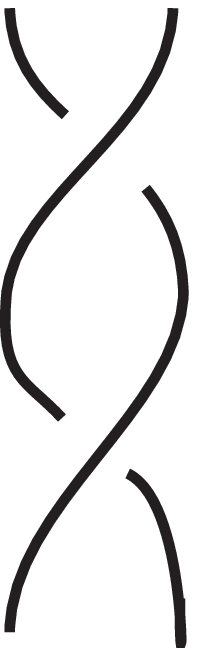}}, involve further interesting results connecting with representation theory.

The invariant associated with a braid $\boldsymbol{b}$ is, roughly speaking, defined by \lq\lq coloring'' the closure of $\boldsymbol{b}$ by the Schr\"{o}dinger module 
as if we were computing the quantum invariant of a (framed) link. 
Such an operation is not allowed in general since ${}_{D(A)}\Mod$ may not be a ribbon category. 
So we will use the (partial) braided trace, introduced below, to define invariants.

\par \smallskip 
\subsection{Partial traces in braided  monoidal categories} 

We briefly recall the concept of duals in a monoidal category $\mathcal{C}=(\mathcal{C}, \otimes , \mathbb{I} , a, r,l)$. 
For an object $X\in \mathcal{C}$ the triple $(X^{\ast}, e_X, n_X)$ consisting of an object $X^{\ast}\in \mathcal{C}$ and morphisms 
$e_X: X^{\ast}\otimes X \longrightarrow \mathbb{I} ,\  n_X: \mathbb{I}\longrightarrow X\otimes X^{\ast}$ in $\mathcal{C}$ is said to be a \textit{left dual} if two compositions 
\begin{align*}
&X\xrightarrow {\ l^{-1}\ } \mathbb{I} \otimes X \xrightarrow{n_X\otimes {\rm id}} (X\otimes X^{\ast})\otimes X \xrightarrow{\ a\ }  X\otimes (X^{\ast}\otimes X) \xrightarrow{{\rm id}\otimes e_X} X\otimes \mathbb{I} \xrightarrow{\ r\ } X,\\[0.1cm]  
&X^{\ast}\xrightarrow {\ r^{-1}\ } X^{\ast}\otimes \mathbb{I} \xrightarrow{{\rm id}\otimes n_X} X^{\ast}\otimes (X\otimes X^{\ast}) \xrightarrow{a^{-1}}  (X^{\ast}\otimes X)\otimes X^{\ast} \xrightarrow{e_X\otimes {\rm id}} \mathbb{I} \otimes X \xrightarrow{\ l\ }X
\end{align*}
are equal to ${\rm id}_X$ and ${\rm id}_{X^{\ast}}$, respectively. 
If all objects in $\mathcal{C}$ have left duals, then the monoidal category is called \textit{left rigid}. 
\par 
From now on, all monoidal categories are assumed to be strict although almost all definitions and results are not needed this assumption. 
\par 
Let $(\mathcal{C}, c)$ is a left rigid braided monoidal category. 
We choose a left dual $(X^{\ast}, e_X, n_X)$ for each object $X\in \mathcal{C}$. 
Let $f: X\otimes Y \longrightarrow X\otimes Z$ be an morphism in $\mathcal{C}$. 
Then the following composition  $\underline{{\rm Tr}}_c^{l, X}(f) : Y\longrightarrow Z$ can be defined: 
\begin{align*}
Y  \xrightarrow{\ l_Y^{-1}\ } 
\mathbb{I} \otimes Y & \xrightarrow{\ n_X\otimes {\rm id}_Y\ } 
X\otimes X^{\ast} \otimes Y \xrightarrow{\ c_{X^{\ast}, X}^{-1}\otimes {\rm id}_Y\ }  
X^{\ast}\otimes X \otimes Y \\ 
& \xrightarrow{ \ {\rm id}\otimes f\ } 
X^{\ast}\otimes X \otimes Z \xrightarrow{\ e_X\otimes {\rm id}_Z\ } 
\mathbb{I} \otimes Z \xrightarrow{\ l_Z\ } Z. 
\end{align*}

The morphism $\underline{{\rm Tr}}_c^{l, X}(f) : Y\longrightarrow Z$ is said to be the \textit{left partial braided trace} of $f$ on $X$. 
Similarly, for a morphism $f: Y\otimes X \longrightarrow Z\otimes X$, the \textit{right partial braided trace} of $f$ on $X$ is defined by the composition 
\begin{align*}
Y  \xrightarrow{\ r_Y^{-1}\ } 
Y\otimes \mathbb{I}  & \xrightarrow{\ {\rm id}_Y\otimes n_X\ } 
Y\otimes X\otimes X^{\ast} \xrightarrow{\ f\otimes {\rm id}\ } 
Z\otimes X \otimes X^{\ast} \\ 
& \xrightarrow{ \ {\rm id}\otimes c_{X, X^{\ast}}\ } 
Z\otimes X^{\ast} \otimes X  \xrightarrow{\ {\rm id}_Z\otimes e_X\ } 
Z\otimes \mathbb{I}\xrightarrow{\ r_Z\ } Z.
\end{align*}

\begin{figure}[htbp]
\centering 
\includegraphics[width=10cm]{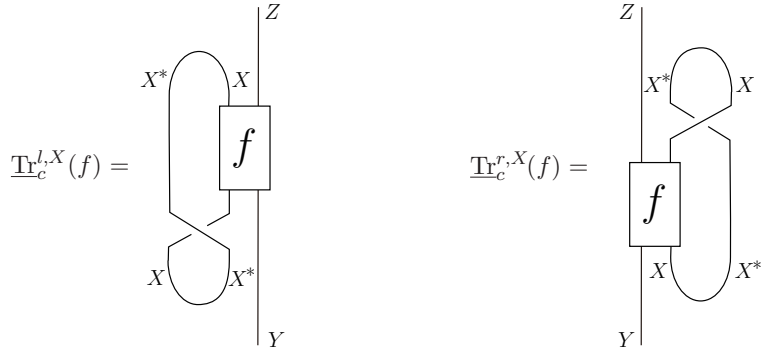}
\caption{the left and right partial traces}\label{Fig1}
\end{figure}

The left and right partial braided traces on $X$ are does not depend on the choice of left duals of $X$. 
\par 
For morphisms $f: X\otimes Y\longrightarrow X\otimes Z$ and $g: Y^{\prime}\longrightarrow Y,\ h: Z\longrightarrow Z^{\prime}$, we have 
\begin{equation}\label{eq4-2-1}
\underline{{\rm Tr}}_c^{l, X}\bigl( ({\rm id}_X\otimes h)\circ f\circ ({\rm id}_X\otimes g)\bigr) =
h\circ \underline{{\rm Tr}}_c^{l, X}(f)\circ g : Y\longrightarrow Z, 
\end{equation}
and for morphisms $f^{\prime}: Y\otimes X\longrightarrow Y\otimes X$ and $g: Y^{\prime}\longrightarrow Y,\ h: Z\longrightarrow Z^{\prime}$, we have 
\begin{equation}\label{eq4-2-2}
\underline{{\rm Tr}}_c^{r, X}\bigl( (h\otimes {\rm id}_X)\circ f^{\prime}\circ (g\otimes {\rm id}_X)\bigr) =
h\circ \underline{{\rm Tr}}_c^{r, X}(f^{\prime})\circ g : Y\longrightarrow Z. 
\end{equation}

\par 
For an endomorphism $f: X\longrightarrow X$ in $(\mathcal{C}, c)$, 
the left braided trace $\underline{{\rm Tr}}_c^{l,X}(f)$ and the right braided trace  $\underline{{\rm Tr}}_c^{r,X}(f)$ are defined by 
$\underline{{\rm Tr}}_c^{l,X}(f):=\underline{{\rm Tr}}_c^{l, X}(f\otimes {\rm id}_{ \mathbb{I}})$ and $\underline{{\rm Tr}}_c^r(f):=\underline{{\rm Tr}}_c^{r, X}({\rm id}_{ \mathbb{I}}\otimes f)$.  
They coincide with the following compositions, respectively. 
\begin{align*}
\underline{{\rm Tr}}_c^l(f) & :   \mathbb{I} \xrightarrow{\ n_X\ }X\otimes X^{\ast}\xrightarrow{c_{X^{\ast}, X}^{-1}}
X^{\ast}\otimes X \xrightarrow{{\rm id}\otimes f} X^{\ast}\otimes X\xrightarrow{\ e_X\ } \mathbb{I}, \\ 
\underline{{\rm Tr}}_c^r (f) & :  \mathbb{I} \xrightarrow{\ n_X\ }X\otimes X^{\ast}\xrightarrow{f\otimes {\rm id}} 
X\otimes X^{\ast}\xrightarrow{c_{X,X^{\ast}}}X^{\ast}\otimes X\xrightarrow{\ e_X\ } \mathbb{I}. 
\end{align*}

\par \smallskip  
\begin{lem}\label{4-3}
Let $(\mathcal{C}, c)$ and $(\mathcal{D}, c')$ be left rigid braided monoidal categories, and 
$(F,\phi ,\omega ):\mathcal{C}\longrightarrow \mathcal{D}$ be a braided monoidal functor. Then
\par 
(1) For any morphism $f: X\otimes Y\longrightarrow X\otimes Z$ in $\mathcal{C}$ 
\begin{equation}\label{eq4-0}
F\bigl( \underline{{\rm Tr}}_c^{l, X}(f)\bigr) =\underline{{\rm Tr}}_{c^{\prime}}^{l, F(X)}\bigl( \phi _{X, Z}^{-1}\circ F(f)\circ \phi _{X,Y}\bigr) .
\end{equation}
\par 
(2)  For any morphism  $f: Y\otimes X\longrightarrow Z\otimes X$ in $\mathcal{C}$ 
\begin{equation}
F\bigl( \underline{{\rm Tr}}_c^{r, X}(f)\bigr) =\underline{{\rm Tr}}_{c^{\prime}}^{r, F(X)}\bigl(\phi _{Z, X}^{-1}\circ F(f)\circ \phi _{Y, X}\bigr) .
\end{equation}
\end{lem}
\begin{proof}
For each $X\in \mathcal{C}$ we choose a left dual $(X^{\ast},e_X, n_X)$. Then 
$(F(X^{\ast}), e_{F(X)}^{\prime} ,n_{F(X)}^{\prime})$ is a left dual of $F(X)$, where 
\begin{align*}
e_{F(X)}^{\prime}&:=\omega ^{-1}\circ F(e_X)\circ \phi _{X^{\ast},X}: F(X^{\ast})\otimes F(X)\longrightarrow  \mathbb{I}^{\prime}, \\ 
n_{F(X)}^{\prime}&:=\phi _{X,X^{\ast}}^{-1}\circ F(n_X)\circ \omega : \mathbb{I}^{\prime}\longrightarrow F(X)\otimes F(X^{\ast}). 
\end{align*}
By using this left dual of $F(X)$ and computing the partial braided trace $\underline{{\rm Tr}}_c^{l, F(X)}\bigl( \phi _{X, Z}^{-1}\circ F(f)\circ \phi _{X,Y}\bigr)$, we have the desired equation (\ref{eq4-0}).
Part (2) is also proved as in the same manner with Part (1). 
\end{proof}

\par \medskip 
Let us describe a relationship between left and right partial traces. 
Let $c$ be a braiding for a monoidal category $\mathcal{C}$. 
Then, the collection  $\bar{c}$ consisting of isomorphisms $\bar{c}_{X,Y}:=c_{Y,X}^{-1}: X\otimes Y\longrightarrow Y\otimes X$ over all pairs $(X,Y)$ of objects in  $\mathcal{C}$ is also a braiding for  $\mathcal{C}$.  

\par 
Let $(\mathcal{C}, c)$ be a left rigid braided monoidal category chosen left duals $(X^{\ast}, e_X, n_X)$ for all objects $X$ in $\mathcal{C}$. 
Then, for two objects $X, Y$  in $\mathcal{C}$ there is a natural isomorphism $j_{X,Y}: Y^{\ast} \otimes X^{\ast} \longrightarrow (X\otimes Y)^{\ast}$ such that $e_{X\otimes Y}\circ (j_{X,Y}\circ {\rm id}_{X\otimes Y})=e_Y\circ ({\rm id}_{Y^{\ast}}\otimes e_X\otimes {\rm id}_Y)$ \cite{FY}. 
For any morphism $f: X\longrightarrow Y$ in $(\mathcal{C}, c)$, 
there is a unique morphism ${}^t\kern-0.1em f: Y^{\ast} \longrightarrow X^{\ast}$ in $\mathcal{C}$, which is characterized by $e_X\circ ({}^t\kern-0.1em f\otimes {\rm id}_X)=e_Y\circ ({\rm id}_{Y^{\ast}}\otimes f)$. 
The morphism ${}^t\kern-0.1em f$ is called the left transpose of $f$. 
The left and right partial traces are related as follows. 

\par \smallskip 
\begin{lem}\label{4-4}
For any morphism $f: X\otimes Y \longrightarrow X\otimes Z$ in a left rigid braided monoidal category $(\mathcal{C}, c)$,  
$$\underline{{\rm Tr}}_c^{r, X^{\ast}}(j_{X,Y}^{-1}\circ {}^t\kern-0.1em f\circ j_{X,Z})={}^t\kern-0.1em \bigl(\underline{{\rm Tr}}_{\bar{c}}^{l,X}(f)\bigr).$$
\end{lem} 
\begin{proof}
The equation of the lemma is obtained from a graphical calculus depicted as in Figure~\ref{Fig2}. 

\begin{figure}[htbp]
\centering 
\includegraphics[width=16cm]{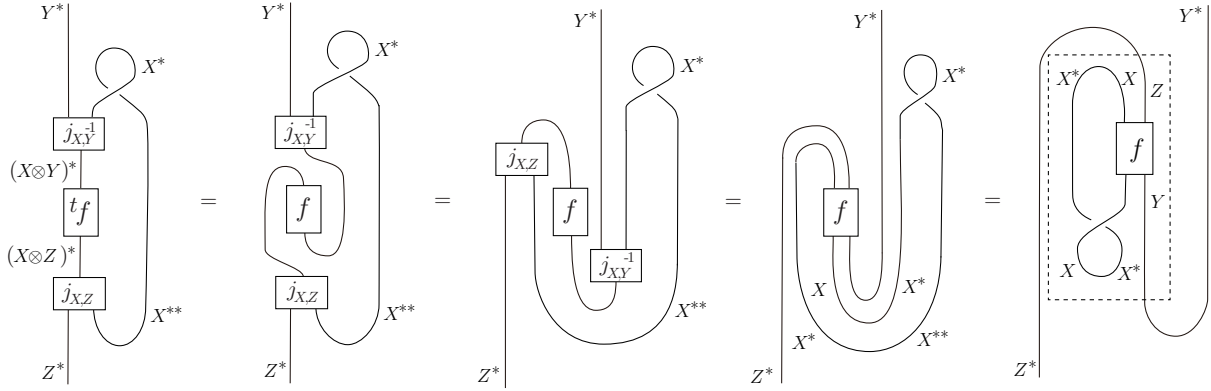}
\caption{a graphical calculus for the proof of Lemma~\ref{4-4}}
\label{Fig2}
\end{figure}
\end{proof} 

\par \medskip 
Let $M$ be an object in a strict left rigid braided monoidal category $(\mathcal{C}, c)$. 
For each endomorphism $f\in {\rm End}(M^{\otimes n})$ and each positive integer $k$ ($1\leq k\leq n$), 
we set 
$\underline{{\rm Tr}}_c^{l, k}(f):=\underline{{\rm Tr}}_c^{l, M^{\otimes k}}(f)$, $\underline{{\rm Tr}}_c^{r, k}(f):=\underline{{\rm Tr}}_c^{r, M^{\otimes k}}(f)$, and 
\begin{equation}\label{eq4-1}
\underline{\widetilde{{\rm Tr}}}_c^l(f) :=
\overbrace{(\underline{{\rm Tr}}_c^{l, 1}\circ \cdots \circ \underline{{\rm Tr}}_c^{l, 1})}^{n}( f) ,\quad 
\underline{\widetilde{{\rm Tr}}}_c^r(f)  :=
\overbrace{(\underline{{\rm Tr}}_c^{r, 1}\circ \cdots \circ \underline{{\rm Tr}}_c^{r, 1})}^{n}( f) . 
\end{equation} 

\par \smallskip 
The modified traces (\ref{eq4-1}) are preserved by a braided monoidal functor. More precisely:  

\par \smallskip 
\begin{prop}\label{4-5}
Let $\mathcal{C}=(\mathcal{C}, c)$ and $\mathcal{D}=(\mathcal{D}, c^{\prime})$ be 
strict left rigid braided monoidal categories, and  
$(F,\phi ,\omega ):\mathcal{C}\longrightarrow \mathcal{D}$ 
be a braided monoidal functor. 
Let $M$ be an object in $\mathcal{C}$, and $k$ be a positive integer, and 
define the isomorphism $\phi ^{(k)} : F(M)^{\otimes k} \longrightarrow F(M^{\otimes k})$ in $\mathcal{D}$ 
by 
$$
\phi ^{(1)}  := {\rm id}_{F(M)},\quad 
\phi ^{(k)} :=\phi _{M, M^{\otimes (k-1)}}\circ ({\rm id}_{F(M)}\otimes \phi ^{(k-1)})\qquad (k\geq 2). 
$$
Then for an endomorphism $f$ on $M^{\otimes n}$ in $\mathcal{C}$, 
the following equations hold. 
\begin{align}
\underline{\widetilde{{\rm Tr}}}_{c^{\prime}}^l \bigl( (\phi ^{(n)})^{-1}\circ F(f)\circ \phi ^{(n)}\bigr) 
& =\omega ^{-1}\circ \bigl( F\bigl( \underline{\widetilde{{\rm Tr}}}_c^l(f)\bigr) \bigr)  \circ \omega , \label{eq4-2} \\ 
\underline{\widetilde{{\rm Tr}}}_{c^{\prime}}^r \bigl( (\phi ^{(n)})^{-1}\circ F(f)\circ \phi ^{(n)}\bigr) 
& =\omega ^{-1}\circ \bigl(F\bigl( \underline{\widetilde{{\rm Tr}}}_c^r(f)\bigr) \bigr)  \circ \omega . \label{eq4-3}
\end{align}
\par 
Therefore, if $\mathcal{C}$, $\mathcal{D}$ and $(F,\phi ,\omega )$ are $\boldsymbol{k}$-linear, then  
$$\underline{\widetilde{{\rm Tr}}}_{c^{\prime}}^l \bigl( (\phi ^{(n)})^{-1}\circ F(f)\circ \phi ^{(n)}\bigr) 
=\underline{\widetilde{{\rm Tr}}}_c^l(f), \quad 
\underline{\widetilde{{\rm Tr}}}_{c^{\prime}}^r \bigl( (\phi ^{(n)})^{-1}\circ F(f)\circ \phi ^{(n)}\bigr) 
= \underline{\widetilde{{\rm Tr}}}_c^r(f)$$
as elements in $\boldsymbol{k}$.  
\end{prop}
\begin{proof}
We set $g:=(\phi ^{(n)})^{-1}\circ F(f)\circ \phi ^{(n)}$. 
By (\ref{eq4-2-1}) and Lemma~\ref{4-3}(1)
we have 
\begin{align*}
\underline{{\rm Tr}}_{c^{\prime}}^{l, 1}(g) 
&=(\phi ^{(n-1)})^{-1}\circ \Bigl( \underline{{\rm Tr}}_{c^{\prime}}^{l, 1}\bigl( \phi _{M, M^{\otimes (n-1)}}^{-1}\circ F(f)\circ \phi _{M, M^{\otimes (n-1)}} \bigr) \Bigr) \circ \phi ^{(n-1)} \\ 
&=(\phi ^{(n-1)})^{-1}\circ F\bigl( \underline{{\rm Tr}}_c^{l, 1}(f)\bigr) \circ \phi ^{(n-1)}. 
\end{align*}

The same arguments for  
$f_1:=\underline{{\rm Tr}}_c^{l, 1}(f)$ and $g_1:=\underline{{\rm Tr}}_{c^{\prime}}^{l, 1}(g)$ provide the equation 
$$
\underline{{\rm Tr}}_{c^{\prime}}^{l, 1}\bigl( \underline{{\rm Tr}}_{c^{\prime}}^{l, 1}(g)\bigr)  =(\phi ^{(n-2)})^{-1}\circ F\bigl( \underline{{\rm Tr}}_c^{l, 1}\bigl( \underline{{\rm Tr}}_c^{l, 1}(f)\bigr) \bigr) \circ \phi ^{(n-2)}. 
$$

By repeating the same arguments, the equation
\begin{equation}\label{eq4-4}
\overbrace{(\underline{{\rm Tr}}_{c^{\prime}}^{l, 1}\circ \cdots \circ \underline{{\rm Tr}}_{c^{\prime}}^{l, 1})}^{n-1}(g)=F\bigl( \overbrace{(\underline{{\rm Tr}}_c^{l, 1}\circ \cdots \circ \underline{{\rm Tr}}_c^{l, 1})}^{n-1}( f)\bigr) 
\end{equation}
is obtained. 
Setting $f_{n-1}:= \overbrace{(\underline{{\rm Tr}}_c^{l, 1}\circ \cdots \circ \underline{{\rm Tr}}_c^{l, 1})}^{n-1}( f)$ and applying $\underline{{\rm Tr}}_{c^{\prime}}^{l, 1}$ to the equation (\ref{eq4-4}), we have the desired equation 
$$\underline{\widetilde{{\rm Tr}}}_{c^{\prime}}^l(g) 
=\underline{{\rm Tr}}_{c^{\prime}}^{l, 1}\Bigl( F\bigl( f_{n-1}\bigr) \Bigr) =\omega ^{-1}\circ F\bigl(\underline{{\rm Tr}}_c^{l, 1} (f_{n-1})\bigr)\circ \omega =\omega ^{-1}\circ F\bigl(\underline{\widetilde{{\rm Tr}}}_c^l(f) \bigr)\circ \omega . 
$$

As in a similar way, the equation (\ref{eq4-3}) can be shown by using $\phi ^{(k)}=\phi _{M^{\otimes (k-1)}, M}\circ (\phi ^{(k-1)}\otimes {\rm id}_{F(M)})$. 
The last two equations are immediately derived from (\ref{eq4-2}) and (\ref{eq4-3}) by setting $\underline{\widetilde{{\rm Tr}}}_c^l(f)=\lambda {\rm id}_{\boldsymbol{k}}$ and $\underline{\widetilde{{\rm Tr}}}_c^r(f)=\lambda ^{\prime}{\rm id}_{\boldsymbol{k}}$ for some $\lambda , \lambda ^{\prime}\in \boldsymbol{k}$.  
\end{proof}

\par \smallskip 
As the same manner of the proof of the above proposition with help from Lemma~\ref{4-4} we have: 

\par \smallskip 
\begin{prop}\label{4-6}
Let $\mathcal{C}=(\mathcal{C}, c)$  be a $\boldsymbol{k}$-linear strict left rigid braided monoidal category. 
Let $M$ be an object in $\mathcal{C}$, and $k$ be a positive integer, and 
define the isomorphism $j^{(k)} : (M^{\ast})^{\otimes k} \longrightarrow (M^{\otimes k})^{\ast}$ in $\mathcal{C}$ by 
$$
j^{(1)}  := {\rm id}_{M^{\ast}},\quad 
j^{(k)}  :=j_{M^{\otimes (k-1)}, M}\circ ({\rm id}_{M^{\ast}}\otimes j^{(k-1)})\qquad (k\geq 2). 
$$
Then for an endomorphism $f$ on $M^{\otimes n}$ in $\mathcal{C}$, 
the following equation holds as elements in $\boldsymbol{k}$:  
\begin{align}
\underline{\widetilde{\rm{Tr}}}_{c}^r \bigl( (j^{(n)})^{-1}\circ {}^t\kern-0.1em f\circ j^{(n)}\bigr) 
& =\underline{\widetilde{\rm{Tr}}}_{\bar{c}}^l(f). 
\end{align}
\end{prop}

\par \smallskip 
Given a quasitriangular Hopf algebra $(A, R)$, a braiding  $c^R=\{ \ c^R_{X,Y} : X\otimes Y \longrightarrow Y\otimes X\ \}_{X,Y\in  {}_A\Mod}$ is defined by 
$$(c^R)_{X,Y}(x\otimes y)=T_{X,Y}\bigl( R\cdot (x\otimes y)\bigr)  $$
for all $x\in X$ and $y\in Y$, where $R\cdot (x\otimes y)$ is the diagonal action of $R$ to $X\otimes Y$. 
Denoted by ${}_{(A,R)}\Mod$ is the braided monoidal category $({}_A\Mod, c^R)$.  
We use the  notation $\underline{{\rm Tr}}_{R}$ instead of $\underline{{\rm Tr}}_{c^R}$. 

\par \smallskip 
\begin{exam}\label{4-7}
Let $(A, R)$ a quasitriangular Hopf algebra over $\boldsymbol{k}$, and $u$ be the Drinfel'd element of it. 
It is well-known that the Drinfel'd element $u$ is invertible, and when $R$ is written as $R=\sum_j\alpha _j\otimes \beta _j$,  the inverse is given by $R^{-1}=(S\otimes {\rm id})(R)=\sum_jS(\alpha _j)\otimes \beta _j$, and 
$u=\sum_j S(\beta _j)\alpha _j$ and $u^{-1}=\sum_j\beta _jS^2(\alpha _j)$ \cite{Dri2, Ra3}.  
\par 
Let $M$ be a finite-dimensional left $A$-module. 
For any $a\in A$ the action of $a$ on $M$ is denoted by $\underline{a}_M$. 
Then for any  $A$-module endomorphism $f$ on $M^{\otimes n}$ the following formulas hold: 
\begin{align}
\underline{\widetilde{\rm{Tr}}}_R^l(f)&={\rm Tr}((\underline{u}^{-1}_M\otimes \cdots \otimes \underline{u}^{-1}_M)\circ f), \label{eq4-5}\\ 
\underline{\widetilde{\rm{Tr}}}_R^r(f)&={\rm Tr}((\underline{u}_M\otimes \cdots \otimes \underline{u}_M)\circ f),  \label{eq4-6}
\end{align}
where $\rm{Tr}$ in the right-hand side stands for the usual trace on linear transformations. 
\end{exam} 
\begin{proof}
Here, we only prove the first equation since the second equation can be proved by the same argument. 
The equation (\ref{eq4-5}) can be shown by induction on $n$ as follows. 
\par 
Let $\{ e_i\}_{i=1}^d$ be a basis for $M$.  For any $a\in A$, 
$a\cdot e_i$ is expressed as 
$a\cdot e_i=\sum_{i^{\prime}=1}^d M_{i^{\prime}, i}(a) \kern0.2em e_{i^{\prime}}$ 
for some $M_{i^{\prime}, i}(a)\in \boldsymbol{k}$. 
Then $\underline{\widetilde{{\rm Tr}}}_R^l(f)=\underline{{\rm Tr}}_R^{l,1}(f)
=\sum M_{i^{\prime}, i}(\beta _j)M_{k, i^{\prime}}(f)M_{i, k}(S^2(\alpha _j)) 
=\sum M_{i^{\prime}, i^{\prime}}(\underline{u}_M^{-1}\circ f) 
={\rm Tr}(\underline{u}_M^{-1}\circ f)$. 
\par 
Next, assume that the equation $\underline{\widetilde{{\rm Tr}}}_R^l(g)={\rm Tr}((\underline{u}^{-1}_M\otimes \cdots \otimes \underline{u}^{-1}_M)\circ g)$ holds for any $A$-module endomorphism $g$ on $M^{\otimes (n-1)}$. 
Let $f$ be an $A$-module endomorphism on $M^{\otimes n}$. 
Then $g:=\underline{{\rm Tr}}_R^{l,1}(f)$ is  an $A$-module endomorphism on $M^{\otimes (n-1)}$. 
Applying the induction hypothesis, we have 
$\underline{\widetilde{{\rm Tr}}}_R^l(f) 
=\underline{\widetilde{{\rm Tr}}}_R^l(g) 
={\rm Tr}(\underbrace{(\underline{u}^{-1}_M\otimes \cdots \otimes \underline{u}^{-1}_M)}_{(n-1)}\circ g) 
={\rm Tr}(\underbrace{(\underline{u}^{-1}_M\otimes \cdots \otimes \underline{u}^{-1}_M)}_{n}\circ f)$. 
\end{proof} 

\par \smallskip 
\subsection{Construction of monoidal Morita invariants} 
In this subsection we introduce a family of monoidal Morita invariants of a finite-dimensional Hopf algebra by using partial braided traces. 
\par 
Let $\mathcal{C}=(\mathcal{C}, c)$ be a strict left rigid braided monoidal category, and $M$ be an object in $\mathcal{C}$.  
Then there is a representation $\rho _M: B_n \longrightarrow {\rm GL}(M^{\otimes n})$ of the $n$-strand braid group $B_n$ 
such that each positive crossing and negative crossing correspond to $c_{M,M}$ and $c_{M,M}^{-1}$, respectively \cite{RT}. 
For each $\boldsymbol{b}\in B_n$ we set 
$$\underline{\boldsymbol{b}\, \text{-}\dim}_c^l(M)  :=
\underline{\widetilde{{\rm Tr}}}_c^l\bigl( \rho _M(\boldsymbol{b})\bigr),\quad  
\underline{\boldsymbol{b}\, \text{-}\dim}_c^r(M)  :=
\underline{\widetilde{{\rm Tr}}}_c^r\bigl( \rho _M(\boldsymbol{b})\bigr) .$$
By Example \ref{4-7},
\begin{equation}
  \label{eq:q-dim-as-braided-dim}
  \underline{\boldsymbol{1}\, \text{-}\dim}_c^r(M) =
  \text{(the quantum dimension of $M$)}
\end{equation}
in the sense of \cite{MajidBook}, where $\boldsymbol{1}$ is the identity element of $B_1$.

\par \smallskip 
\begin{lem}\label{4-8}
Let $M$ and $N$ be two objects in $\mathcal{C}$. 
\par 
(1) If $M$ and $N$ are isomorphic, then 
$\underline{\boldsymbol{b}\, \text{-}\dim}_c^l(M)=\underline{\boldsymbol{b}\, \text{-}\dim}_c^l(N)$, \ 
$\underline{\boldsymbol{b}\, \text{-}\dim}_c^r(M)=\underline{\boldsymbol{b}\, \text{-}\dim}_c^r(N)$. 
\par 
(2) 
$\underline{\boldsymbol{b}\, \text{-}\dim}_c^r(M^{\ast})=\underline{\boldsymbol{b}\, \text{-}\dim}_{\bar{c}}^l(M)$. 
\end{lem}
\begin{proof}
(1) Let $\varphi : M\longrightarrow N$ be an isomorphism.  
The map $\varphi ^{\otimes n}: M^{\otimes n} \longrightarrow N^{\otimes n}$ is also an isomorphism.  
Let $\rho_M: B_n\longrightarrow {\rm GL}(M^{\otimes n})$ and $\rho _N: B_n\longrightarrow {\rm GL}(N^{\otimes n})$ be the representations induced from the braiding $c$. 
Since 
$c_{N,N}\circ (\varphi \otimes \varphi )= (\varphi \otimes \varphi )\circ c_{M,M}$ from naturality of $c$, 
the endomorphisms $f:=\rho _M(\boldsymbol{b})$ and $g:=\rho _N(\boldsymbol{b})$ satisfy 
$g\circ \varphi ^{\otimes n}=\varphi ^{\otimes n}\circ f$. 
Thus, $g$ is expressed as $g=(\varphi ^{\otimes n})\circ f\circ (\varphi ^{\otimes n})^{-1}$, and 
it follows from Proposition~\ref{4-5} that 
$\underline{\boldsymbol{b}\, \text{-}\dim}_c^l(N)
=\underline{\widetilde{{\rm Tr}}}_c^l(g) 
=\underline{\widetilde{{\rm Tr}}}_c^l(f) 
=\underline{\boldsymbol{b}\, \text{-}\dim}_c^l(M)$. 
The equation $\underline{\boldsymbol{b}\, \text{-}\dim}_c^r(M)=\underline{\boldsymbol{b}\, \text{-}\dim}_c^r(N)$ is also shown by the same argument. 
\par 
(2) By the definition of the natural isomorphism $j_{M,N}: N^{\ast}\otimes M^{\ast}\longrightarrow (M\otimes N)^{\ast}$,  it is easy to see that 
$j_{M,N}\circ c_{M^{\ast}, N^{\ast}}={}^t\kern-0.1em (c_{M,N})\circ j_{N,M}$. 
It follows that the representation $\rho_{M^{\ast}} : B_n\longrightarrow {\rm GL}((M^{\ast})^{\otimes n})$ induced from the braiding $c$ satisfies 
$j^{(n)}\circ \rho_{M^{\ast}}(\boldsymbol{b})={}^t\kern-0.1em \bigl(\rho_M(\boldsymbol{b})\bigr)\circ j^{(n)}$, 
where $j^{(n)}$ is the isomorphism defined in Proposition~\ref{4-6}.   
Thus we have 
$\underline{\boldsymbol{b}\, \text{-}\dim}_c^r(M^{\ast})=
\underline{\widetilde{{\rm Tr}}}_c^r\bigl( (j^{(n)})^{-1}\circ {}^t\kern-0.1em \bigl(\rho_M(\boldsymbol{b})\bigr)\circ j^{(n)}\bigr) 
=\underline{\widetilde{{\rm Tr}}}_{\bar{c}}^l\bigl(\rho_M(\boldsymbol{b})\bigr) =\underline{\boldsymbol{b}\, \text{-}\dim}_{\bar{c}}^l(M)$.  
\end{proof} 

\par 
Let $(A, R)$ be a quasitriangular Hopf algebra over $\boldsymbol{k}$, and $M$ be a finite-dimensional left $A$-module. 
For each $\boldsymbol{b}\in B_n$ we set 
$$\underline{\boldsymbol{b}\, \text{-}\dim}_R^l(M)  :=
\underline{\boldsymbol{b}\, \text{-}\dim}_{c^R}^l(M),\quad  
\underline{\boldsymbol{b}\, \text{-}\dim}_R^r(M)  :=
\underline{\boldsymbol{b}\, \text{-}\dim}_{c^R}^l(M).$$
In particular, in the case where $(A,R)=(D(H), \mathcal{R})$ for some finite-dimensional Hopf algebra $H$ over $\boldsymbol{k}$, 
we denote 
$$
\underline{\boldsymbol{b}\, \text{-}\dim}^l(M)  :=
\underline{\boldsymbol{b}\, \text{-}\dim}_{\mathcal{R}}^l(M) ,\quad 
\underline{\boldsymbol{b}\, \text{-}\dim}^r(M)  :=
\underline{\boldsymbol{b}\, \text{-}\dim}_{\mathcal{R}}^r(M).  
$$

Now we will show that $\underline{\boldsymbol{b}\, \text{-}\dim}^l({}_{D(H)}H)$ and $\underline{\boldsymbol{b}\, \text{-}\dim}^r({}_{D(H)}H)$ are preserved under $\boldsymbol{k}$-linear monoidal equivalences of  the module categories. 

\par \smallskip 
\begin{thm}\label{4-9}
Let $A$ and $B$ be finite-dimensional Hopf algebras over $\boldsymbol{k}$. 
If ${}_A\Mod$ and ${}_B\Mod$ are equivalent as $\boldsymbol{k}$-linear monoidal  categories, 
then  $\underline{\boldsymbol{b}\, \text{-}\dim}^r({}_{D(A)}A)=\underline{\boldsymbol{b}\, \text{-}\dim}^r({}_{D(B)}B)$ for all 
$\boldsymbol{b}\in B_n$.  
The same statement holds for $\underline{\boldsymbol{b}\, \text{-}\dim}^l$. 
\end{thm}
\begin{proof}
Let $(F, \phi , \omega ) : ({}_A\Mod, c) \longrightarrow ({}_B\Mod, c^{\prime})$ be a $\boldsymbol{k}$-linear braided monoidal functor. 
Then $\underline{\boldsymbol{b}\, \text{-}\dim}^r_{c^{\prime}} F(M)=\underline{\boldsymbol{b}\, \text{-}\dim}^r_{c} M$ for all finite-dimensional left $A$-module $M$ and $\boldsymbol{b}\in B_n$. 
This equation can be shown as follows. 
Since $\phi _{M,M}\circ c^{\prime}_{F(M), F(M)} =F(c_{M,M})\circ \phi _{M,M}$, and 
$f:=\rho _M(\boldsymbol{b}): M^{\otimes n}\longrightarrow M^{\otimes n}$ is expressed by compositions and tensor products of  $c_{M,M}^{\pm 1}$ and ${\rm id}_M$, the isomorphism $\phi ^{(n)}: F(M)^{\otimes n}\longrightarrow F(M^{\otimes n})$ defined in Proposition~\ref{4-4} satisfies 
$\phi ^{(n)}\circ \rho _{F(M)}(\boldsymbol{b})=F(f)\circ \phi ^{(n)}$. 
Therefore, by Proposition~\ref{4-4} we have 
\begin{align*}
\underline{\boldsymbol{b}\, \text{-}\dim}^r_{c^{\prime}} \bigl( F(M)\bigr) 
&=\underline{\widetilde{{\rm Tr}}}_{c^{\prime}}^r \bigl( (\phi ^{(n)})^{-1}\circ F( f)\circ \phi ^{(n)}\bigr) 
=\underline{\widetilde{{\rm Tr}}}_c^r(f ) 
=\underline{\boldsymbol{b}\, \text{-}\dim}^r_{c} (M). 
\end{align*}
\indent 
Suppose that $(F, \phi , \omega ) : {}_A\Mod\longrightarrow {}_B\Mod$
is an equivalence of $\boldsymbol{k}$-linear monoidal categories. 
By Corollary~\ref{3-5},  
$\tilde{F}({}_{D(A)}A)\cong {}_{D(B)}B$ as left $D(B)$-modules. 
It follows from Lemma~\ref{4-8} that $\underline{\boldsymbol{b}\, \text{-}\dim}^r_{\mathcal{R}}({}_{D(A)}A) 
=\underline{\boldsymbol{b}\, \text{-}\dim}^r_{{\mathcal{R}}^{\prime}}\tilde{F}({}_{D(A)}A) 
=\underline{\boldsymbol{b}\, \text{-}\dim}^r_{{\mathcal{R}}^{\prime}}({}_{D(B)}B)$, 
where $\mathcal{R}$ and $\mathcal{R}^{\prime}$ are the canonical universal $R$-matrices of $D(A)$ and $D(B)$, respectively. 
\end{proof}

\par \medskip 
By using Theorem~\ref{3-10} 
the monoidal Morita invariants $\underline{\boldsymbol{b}\, \text{-}\dim }^l$ and $\underline{\boldsymbol{b}\, \text{-}\dim }^r$ of the dual Schr\"{o}dinger modules ${}_{D(A)}A^{\ast{\rm cop}}$ and ${}_{D(A^{\ast})}(A^{\ast})^{\ast{\rm cop}}$ are computable from the monoidal Morita invariants of the Schr\"{o}dinger modules ${}_{D(A^{\ast})}(A^{\ast})$ and ${}_{D(A)}A$, respectively. 

\par \bigskip 
\begin{prop}\label{4-13}
Let $A$ be a finite-dimensional Hopf algebra over $\boldsymbol{k}$, and 
$\mathcal{R}, \mathcal{R}^{\prime}$ be the canonical universal $R$-matrices of $D(A), D(A^{\ast})$, respectively. 
For any $\boldsymbol{b}\in B_n$, the following equations hold. 
\par 
(1) $\underline{\boldsymbol{b}\, \text{-}\dim }^l({}_{D(A^{\ast})}(A^{\ast})^{\ast{\rm cop}})
=  \underline{\boldsymbol{b}\, \text{-}\dim }^l( {}_{D(A)}A),\  
\underline{\boldsymbol{b}\, \text{-}\dim }^r ({}_{D(A^{\ast})}(A^{\ast})^{\ast{\rm cop}})
=  \underline{\boldsymbol{b}\, \text{-}\dim }^r ( {}_{D(A)}A)$.  
\par 
(2) $\underline{\boldsymbol{b}\, \text{-}\dim }^l ({}_{D(A^{\ast})}A^{\ast})
=\underline{\boldsymbol{b}\, \text{-}\dim }^l ({}_{D(A)}A^{\ast{\rm cop}})$,\ 
$\underline{\boldsymbol{b}\, \text{-}\dim }^r ({}_{D(A^{\ast})}A^{\ast})
=\underline{\boldsymbol{b}\, \text{-}\dim }^r ({}_{D(A)}A^{\ast{\rm cop}})$. 
\end{prop} 
\begin{proof} 
(1) Let $F_A: ({}_{D(A^{\ast})}\Mod, c^{\mathcal{R}^{\prime}})^{{\rm rev}}\longrightarrow ({}_{D(A)}\Mod, c^{\mathcal{R}})$ be the equivalence of braided monoidal categories defined in Theorem~\ref{3-10}. 
Setting $M:={}_{D(A^{\ast})}(A^{\ast})^{\ast{\rm cop}}$, we have 
$\underline{\boldsymbol{b}\, \text{-}\dim }^r_{\mathcal{R}} (F_A(M))=
\underline{\boldsymbol{b}\, \text{-}\dim }^r_{\mathcal{R}^{\prime}}(M)$
from the proof of Theorem~\ref{4-9}. 
Since $F_A(M)$ and ${}_{D(A)}A$ are isomorphic as left $D(A)$-modules by Proposition~\ref{2-7}, 
it follows from Lemma~\ref{4-8}(1) that 
$\underline{\boldsymbol{b}\, \text{-}\dim }^r_{\mathcal{R}} (F_A(M))
=\underline{\boldsymbol{b}\, \text{-}\dim }^r_{\mathcal{R}} ({}_{D(A)}A)$. 
Thus, 
the second equation is obtained. 
Similarly,  the equation \newline $\underline{\boldsymbol{b}\, \text{-}\dim }^l({}_{D(A^{\ast})}(A^{\ast})^{\ast{\rm cop}})
=  \underline{\boldsymbol{b}\, \text{-}\dim }^l( {}_{D(A)}A)$ can be proved. 
\par 
Part (2) can be proved as in the proof of (1) by using Corollary~\ref{2-8} instead of Proposition~\ref{2-7}. 
\end{proof} 

Let $A$ be a finite-dimensional Hopf algebra. 
In view of Example~\ref{4-7}, it is important to know the action of the Drinfel'd element $u \in D(A)$ on a given $D(A)$-module $M$ to compute the braided dimension of $M$. 
Below we give formulas for the actions of $u$ and $S(u)$ on the Schr\"{o}dinger module ${}_{D(A)}A$.

Recall that a {\em left integral} in $A$ is an element $\Lambda \in A$ such that $a \Lambda = \varepsilon(a) \Lambda$ for all $a \in A$. 
A {\em right integral} in $A$ is a left integral in $A^{\rm op}$. 
It is known that a non-zero left integral $\Lambda \in A$ always exists (under our assumption that $A$ is finite-dimensional), and is unique up to a scalar multiple. 
Hence one can define $\alpha \in A^{\ast}$ by $\Lambda a = \langle \alpha, a \rangle \Lambda$ for $a \in A$.
The map $\alpha$ is in fact an algebra map, and does not depend on the choice of $\Lambda$. 
We call $\alpha$ the {\em distinguished grouplike element} of $A^{\ast}$.
The Hopf algebra $A$ is said to be {\em unimodular} if the distinguished grouplike element $\alpha \in A^{\ast}$ is the counit of $A$, or, equivalently, $\Lambda \in A$ is central. 

\begin{lem}
  \label{lem:Sch-mod-u-Su-action}
  With the above notations, we have
  \begin{equation*}
    u \rightharpoonup a = \sum S^2(a_{(1)}) \big \langle \alpha^{-1}, a_{(2)} \big \rangle
    \text{\quad and \quad}
    S(u) \rightharpoonup a = S^{-2}(a)
  \end{equation*}
  for all $a \in {}_{D(A)}A$, where $\alpha^{-1} = \alpha \circ S$.
\end{lem}
\begin{proof}
  Recall that $u = \sum_i S^{-1}_{A^{\ast}}(e_i^{\ast}) \bowtie e_i$, 
where $\{ e_i \}$ is a fixed basis of $A$, and $\{ e_i^{\ast} \}$ is the dual basis to $\{ e_i\}$. 
We first compute the action of $S(u)$. 
Since $S^2(u) = u$,
  \begin{equation*}
    S(u) = S^{-1}(u) = \sum_i (\varepsilon \bowtie S_A^{-1}(e_i)) \cdot (e_i^{\ast} \bowtie 1).
  \end{equation*}
  Hence, for all $a \in {}_{D(A)}A$, we have
  \begin{align*}
    S(u) \rightharpoonup a
    & = \sum_i \big \langle e_i^{\ast}, S^{-1}(a_{(1)}) \big \rangle \, \big (S^{-1}(e_i) \rightharpoonup a_{(2)} \big) \\
    & = \sum_i \big \langle e_i^{\ast}, S^{-1}(a_{(1)}) \big \rangle \, S^{-1}(e_{i(2)}) a_{(2)} e_{i(1)} \\
    & = \sum S^{-1} \Big(S^{-1}(a_{(1)})_{(2)}\Big) a_{(2)} S^{-1}(a_{(1)})_{(1)}
    = S^{-2}(a).
  \end{align*}
  Next, we compute the action of $u$. Fix a non-zero right integral $\lambda \in A^{\ast}$, and define $g \in A$ to be the unique element such that $p \lambda = \langle p, g \rangle \lambda$ for all $p \in A^{\ast}$ 
({\it i.e.}, the distinguished grouplike element of $(A^{\rm cop})^{\ast\ast} = (A^{\ast\,{\rm op}})^{\ast}$ regarded as an element of $A$). 
Radford showed in \cite{Rad1994} that $D(A)$ is unimodular, and $g \bowtie \alpha$ is the distinguished grouplike element of $D(A)^{\rm cop}$. 
Hence, by \cite[Theorem 2]{Ra3}, we have $u \rightharpoonup a = S(u) \rightharpoonup (g \bowtie \alpha) \rightharpoonup a$ for all $a \in {}_{D(A)}A$. 
Using the formula of the fourth power of the antipode \cite{Ra0}, we obtain
  \begin{equation*}
    u \rightharpoonup a
    = \sum \langle \alpha, S^{-1}(a_{(1)}) \rangle \, S^{-2}(g a_{(2)} g^{-1})
    = \sum S^2(a_{(1)}) \big \langle \alpha^{-1}, a_{(2)} \big \rangle.
    \qedhere
  \end{equation*}
\end{proof}

Combining Example~\ref{4-7} and Lemma~\ref{lem:Sch-mod-u-Su-action}, we obtain the following proposition:

\begin{prop}
  \label{prop:braided-dim-involutory-unimodular}
  Let $A$ be a finite-dimensional Hopf algebra over $\boldsymbol{k}$. 
If $A$ is involutory ({\it i.e.} the square of the antipode is the identity) and unimodular, then we have
  \begin{equation*}
    \underline{\boldsymbol{b}\, \text{-}\dim }^l ({}_{D(A)}A)
    = \underline{\boldsymbol{b}\, \text{-}\dim }^r ({}_{D(A)}A)
    = \mathrm{Tr}(\rho(\boldsymbol{b}))
  \end{equation*}
  for all $\boldsymbol{b} \in B_n$, where $\rho: B_n \longrightarrow {\rm GL}(({}_{D(A)}A)^{\otimes n})$ is the braid group action.
\end{prop}

\par \smallskip 
\subsection{Examples} 

We denote by $\sigma_i \in B_n$ ($i = 1, \dots, n - 1$) the braid of $n$ strands with only one positive crossing between the $i$-th and the $(i+1)$-st strands. 
For integers $p$ and $q$ with $p \ge 2$, the braid
\begin{equation*}
  \boldsymbol{t}_{p, q} := (\sigma_1 \sigma_2 \cdots \sigma_{p-1})^q \in B_{p}
\end{equation*}
is called the $(p, q)$-torus braid, as its closure is the $(p, q)$-torus link. 
The below is an example of the computation of the braided dimension associated with $\boldsymbol{b} = \boldsymbol{t}_{2,q}$.

\par \medskip 
\begin{lem}\label{4-10} 
Let $(A, R)$ be a quasitriangular Hopf algebra over $\boldsymbol{k}$, and $u$ be the Drinfel'd element of it. 
For each non-negative integer $m$ and finite-dimensional left $A$-module $X$, 
\begin{align*}
\underline{\boldsymbol{t}_{2,q}\, \text{-}\dim}_{R}^l X
& ={\begin{cases}
\sum {\rm Tr}\bigl(\underline{u^{m-1}(u^{-m})_{(1)}}_X){\rm Tr}(\underline{u^{m-1}(u^{-m})_{(2)}}_X)  & \text{if $q=2m$},\\[0.3cm]  
\sum {\rm Tr}\bigl(\underline{\bigl( u^{m-1}\otimes u^{m-1}\bigr)\Delta (u^{-m})R_{21}}_{X\otimes X}\circ T_{X,X}\bigr)  & \text{if $q=2m+1$}, 
\end{cases}}\\ 
\underline{\boldsymbol{t}_{2,q}\, \text{-}\dim}_{R}^r \kern0.1em X
&=\begin{cases} 
\sum{\rm Tr}(\underline{u^{m+1}(u^{-m})_{(1)}}_X)
{\rm Tr}(\underline{u^{m+1}(u^{-m})_{(2)}}_X) & \text{if $q=2m$},\\[0.3cm]  
\sum {\rm Tr}(\underline{(u^{m+1}\otimes u^{m+1})\Delta (u^{-m})R_{21}}_{X\otimes X}\circ T_{X,X}) & \text{if $q=2m+1$}. 
\end{cases} 
\end{align*}
Here, for elements $a, b\in A$ the notation 
$\underline{a\otimes b}_{X\otimes X}$ stands for the left action on $X\otimes X$ defined by $x\otimes y\longmapsto (a\cdot x)\otimes (b\cdot y)$ for all $x, y\in X$. 
\end{lem}
\begin{proof}
The formula for $\underline{\boldsymbol{t}_{2,q}\, \text{-}\dim}_{R}^l \kern0.1em X$ can be obtained as follows. 
\par 
Let $\{ e_s\} _{s=1}^d$ be a basis for $X$, and $\{ e_s^{\ast}\} _{s=1}^d$ be its dual basis. 
Let $R^{(q)}$ be the element in $A\otimes A$ defined by 
$$R^{(q)}=\begin{cases} 
(R_{21}R)^m & \text{if $q=2m$},\\ 
(R_{21}R)^mR_{21} & \text{if $q=2m+1$}. 
\end{cases}$$
By Example~\ref{4-7}, we see that 
\begin{equation}\label{eq4-8}
\underline{\boldsymbol{t}_{2,q}\, \text{-}\dim}_{R}^l \kern0.1em X
=\begin{cases}
{\rm Tr}(\underline{( u^{-1}\otimes u^{-1}) R^{(q)}}_{X\otimes X})
& \text{if $q$ is even}, \\ 
{\rm Tr}(\underline{( u^{-1}\otimes u^{-1}) R^{(q)}}_{X\otimes X}\circ T_{X,X})
& \text{if $q$ is odd}.  
\end{cases}
\end{equation}
Since $R_{21}R=\Delta (u^{-1})(u\otimes u)=(u\otimes u)\Delta (u^{-1})$ \cite{Dri2}, it follows that 
$(R_{21}R)^m
=(u^m\otimes u^m)\Delta (u^{-m})$. 
Substituting this equation to (\ref{eq4-8}) we obtain the formula for $\underline{\boldsymbol{t}_{2,q}\, \text{-}\dim}_{R}^l \kern0.1em X$ in the lemma. 
By a similar consideration, the formula for $\underline{\boldsymbol{t}_{2,q}\, \text{-}\dim}_{R}^r \kern0.1em X$ can be obtained. 
\end{proof}

In the case where $A$ is semisimple, the braided dimension of the Schr\"{o}dinger module associated with $\boldsymbol{t}_{2,2}$ has the following representation-theoretic meaning:

\begin{thm}
  \label{thm:t-2-2-braid-and-irrep}
  Suppose that $\boldsymbol{k}$ is an algebraically closed field of characteristic zero. 
If $A$ is a finite-dimensional semisimple Hopf algebra over $\boldsymbol{k}$, then
  \begin{equation*}
    \underline{\boldsymbol{t}_{2,2}\text{-}\dim^l}({}_{D(A)}A) = 
    \underline{\boldsymbol{t}_{2,2}\text{-}\dim^r}({}_{D(A)}A) = \dim(A) \, \sharp {\rm Irr}(A),
  \end{equation*}
  where $\sharp {\rm Irr}(A)$ is the number of isomorphism classes of irreducible $A$-modules.
\end{thm}
\begin{proof}
  It is sufficient to show $\underline{\boldsymbol{t}_{2,2}\text{-}\dim^l}({}_{D(A)}A) = \dim(A) \, \sharp {\rm Irr}(A)$ in view of Proposition~\ref{prop:braided-dim-involutory-unimodular}. 
By the assumption, Radford induction functor $I_A: {}_A\boldsymbol{\sf M} \longrightarrow {}_{D(A)}\boldsymbol{\sf M}$ is isomorphic to the functor $D(A) \otimes_A (-)$ by \cite[Lemma 2.3]{HuZhang2}. 
Combining this fact with Proposition~\ref{2-10}, we have
  \begin{equation*}
    ({}_{D(A)}A) \otimes ({}_{D(A)} A) \cong I^A(A_{\rm ad}) \cong D(A) \otimes_A A_{\rm ad}.
  \end{equation*}
  Hence, by Lemma~\ref{4-10},
  \begin{equation}
    \label{eq:Thm-4-13-proof-1}
    \begin{aligned}
   \underline{\boldsymbol{t}_{2,2}\text{-}\dim^l}({}_{D(A)}A)
    & = \sum {\rm Tr}(\underline{(u^{-1})_{(1)}}_{{}_{D(A)}A}) {\rm Tr}(\underline{(u^{-1})_{(2)}}_{{}_{D(A)}A}) \\
    & = {\rm Tr}(\underline{u^{-1}}_{({}_{D(A)}A) \otimes ({}_{D(A)}A)})
    = {\rm Tr}(\underline{u^{-1}}_{D(A) \otimes_A A_{\rm ad}}).      
    \end{aligned}
  \end{equation}

  We use some results on the Frobenius-Schur indicator \cite{LM}. 
Let $V$ be a finite-dimensional left $A$-module. 
The \lq\lq third formula'' \cite[\S6.4]{KSZ} of the $n$-th Frobenius-Schur indicator $\nu_n(V)$ ($n = 1, 2, \ldots$) expresses $\nu_n(V)$ by using the Drinfel'd element, as
  \begin{equation*}
    \nu_n(V) = \frac{1}{\dim(A)} {\rm Tr}(\underline{u^n}_{D(A) \otimes_A V}).
  \end{equation*}
  Since $u$ is of finite order \cite{EG1}, $\dim(A) \, \nu_n(V) \in \mathbb{Z}[\xi]$ ($\subset \boldsymbol{k}$), where $\xi \in \boldsymbol{k}$ is a root of unity of the same order as $u$. Hence, if we denote by $z \mapsto \overline{z}$ the ring automorphism of $\mathbb{Z}[\xi]$ defined by $\xi \mapsto \xi^{-1}$, then we have
  \begin{equation*}
    {\rm Tr}(\underline{u^{-n}}_{D(A) \otimes_A V}) = \overline{\dim(A) \, \nu_n(V)}.
  \end{equation*}
  On the other hand, the \lq\lq first formula'' \cite[\S2.3]{KSZ} yields $\nu_1(V) = \dim \bigl({\rm Hom}_A(\boldsymbol{k}, V)\bigr)$. 
Considering the case where $V$ is the adjoint representation $A_{\rm ad}$, we obtain
  \begin{equation*}
    {\rm Tr}(\underline{u^{-1}}_{D(A) \otimes_A A_{\rm ad}})
    = \overline{\dim(A) \, \nu_1(A_{\rm ad})}
    = \overline{\dim \bigl({\rm Hom}(\boldsymbol{k}, A_{\rm ad})\bigr)}
    = \dim(A) \, \sharp {\rm Irr}(A).
  \end{equation*}
  Now the result follows from~(\ref{eq:Thm-4-13-proof-1}).
\end{proof}

As this theorem suggests, the Schr\"{o}dinger module ${}_{D(A)}A$ has much information about the category of $A$-modules, at least, in the semisimple case. 
However, the computation of the braided dimension is not easy in general. 
Fortunately, if $A$ is a group algebra, then the braided dimension of ${}_{D(A)}A$ closely relates to the link group of the closure of the braid, and can be computed in the following way:

\begin{thm}
  \label{thm:br-dim-kG}
  Let $\boldsymbol{b} \in B_n$. If $A = \boldsymbol{k}[G]$ is the group algebra of a finite group $G$, then
  \begin{equation*}
    \underline{\boldsymbol{b}\, \text{-}\dim }^l ({}_{D(A)}A)
    = \underline{\boldsymbol{b}\, \text{-}\dim }^r ({}_{D(A)}A)
    = \sharp {\rm Hom}(\pi_1(\mathbb{R}^3 \setminus \widehat{\boldsymbol{b}}), G)
  \end{equation*}
  in $\boldsymbol{k}$, where $\widehat{\boldsymbol{b}}$ is the link obtained by closing the braid $\boldsymbol{b}$, and $\pi_1$ means the fundamental group.
\end{thm}
\begin{proof}
Set $X = {}_{D(A)}A$ for simplicity. Then the braiding $c_{X,X}$ is given by
  \begin{equation*}
    c_{X,X}(g \otimes h) = h \otimes (h^{-1} \blacktriangleright g)
    = h \otimes h^{-1} g h \quad (g, h \in G).
  \end{equation*}
  Let $B_n$ act on $G^n$ by
  \begin{equation*}
    \varrho(\sigma_i)(g_1, \dots, g_n) = (g_1, \dots, g_{i-1}, g_{i+1}, g_{i+1}^{-1} g_i g_{i+1}, g_{i+2}, \dots, g_n)
    \quad (g_1, \dots, g_n \in G).
  \end{equation*}
  By Proposition~\ref{prop:braided-dim-involutory-unimodular}, $\underline{\boldsymbol{b}\, \text{-}\dim }^l ({}_{D(A)}A)$ and $\underline{\boldsymbol{b}\, \text{-}\dim }^r ({}_{D(A)}A)$ are equal to the number of fixed points of $\varrho(\boldsymbol{b})$ regarded as an element of $\boldsymbol{k}$. 
On the other hand, the number of fixed points of $\varrho(\boldsymbol{b})$ has been studied by Freyd and Yetter \cite{FY89} in relation with link invariants arising from crossed $G$-sets. 
The claim of this theorem follows from \cite[Proposition 4.2.5]{FY89}.
\end{proof}

\begin{exam}
  \rm
  We consider the case where $A = \boldsymbol{k}[G]$ is the group algebra of a finite group $G$. 
If $\boldsymbol{k}$ is an algebraically closed field of characteristic zero, then we obtain
 \begin{equation}
    \label{eq:concluding-remark-1}
    \begin{gathered}
      \underline{\boldsymbol{t}_{2,2}\,\text{-}\dim}^{l}({}_{D(A)} A)
      = \underline{\boldsymbol{t}_{2,2}\,\text{-}\dim}^{r}({}_{D(A)} A)
      = |G| \cdot \sharp {\rm Conj}(G), \\
      \underline{\boldsymbol{t}_{2,2}\,\text{-}\dim}^{l}({}_{D(A^{\ast})} A^{\ast})
      = \underline{\boldsymbol{t}_{2,2}\,\text{-}\dim}^{r}({}_{D(A^{\ast})} A^{\ast})
      = |G|^2
    \end{gathered}
  \end{equation}
  by Theorem~\ref{thm:t-2-2-braid-and-irrep}, where ${\rm Conj}(G)$ is the set of conjugacy classes of $G$. In particular,
  \begin{equation*}
    \underline{\boldsymbol{t}_{2,2}\,\text{-}\dim}^{l}({}_{D(A)} A)
    \ne \underline{\boldsymbol{t}_{2,2}\,\text{-}\dim}^{l}({}_{D(A^{\ast})} A^{\ast})
  \end{equation*}
  whenever $G$ is non-abelian. 
This result is interesting from the viewpoint that some other monoidal Morita invariants, such as ones introduced in \cite{EG1} and \cite{Shimizu}, are in fact invariants of the braided monoidal category of the representations of the Drinfel'd double.

  In topology, the link $H = \widehat{\boldsymbol{t}}_{2,2}$ is known as the Hopf link. 
Since $\pi_1(\mathbb{R} \setminus H)$ is the free abelian group of rank two, we have 
  \begin{equation}
    \label{eq:concluding-remark-2}
    \underline{\boldsymbol{t}_{2,2}\,\text{-}\dim}^{l}({}_{D(A)} A)
    = \underline{\boldsymbol{t}_{2,2}\,\text{-}\dim}^{r}({}_{D(A)} A) = \sharp {\rm Comm}(G)
  \end{equation}
  by Theorem~\ref{thm:br-dim-kG}, where ${\rm Comm}(G) = \{ (x, y) \in G \times G \mid x y = y x \}$. 
Comparing \eqref{eq:concluding-remark-1} with \eqref{eq:concluding-remark-2}, we get $|G| \cdot \sharp {\rm Conj}(G) = \sharp {\rm Comm}(G)$. 
Although this formula itself is well-known in finite group theory, 
we expect that some non-trivial formulas for finite groups (or, more generally, for finite-dimensional semisimple Hopf algebras) would be obtained via the investigation of the braided dimension.
\end{exam}

By (\ref{eq:q-dim-as-braided-dim}) and \cite[Example 9.3.8]{MajidBook}, we have $\underline{\boldsymbol{1}\,\text{-}\dim}^{r}({}_{D(A)}A) = {\rm Tr}(S_A^2)$ (see \cite{BulacuTorrecillas} for the quasi-Hopf case). 
In particular, $\underline{\boldsymbol{1}\,\text{-}\dim}^{r}({}_{D(A)}A) = 0$ whenever $A$ is not cosemisimple. 
More strongly, we have the following theorem:

\par \smallskip 

\begin{thm}
  \label{thm:br-dim-non-coss}
  Let $A$ be a finite-dimensional Hopf algebra. 
If $A$ is not cosemisimple, then we have
  $\underline{\boldsymbol{b}\, \text{-}\dim }^l ({}_{D(A)}A) = \underline{\boldsymbol{b}\, \text{-}\dim }^r ({}_{D(A)}A) = 0$
  for all braids $\boldsymbol{b}$.
\end{thm}
\begin{proof}
  Let, in general, $X$ be a finite-dimensional Hopf algebra, let $\Lambda \in X$ be a left integral, and let $\lambda \in X^{\ast}$ be the right integral such that $\langle \lambda, \Lambda \rangle = 1$. 
By \cite[Proposition 2 (a)]{Rad1994},
  \begin{equation*}
    \mathrm{Tr}\Big( X \longrightarrow X; \ x \mapsto S^2(x_{(2)}) \langle p, x_{(1)} \rangle \Big) = \langle \lambda, 1 \rangle \langle p, \Lambda \rangle
  \end{equation*}
  for all $p \in X^{\ast}$. 
Applying this formula to $X = A^{\rm op\, cop}$, we have
  \begin{equation}
    \label{eq:trace-formula}
    \mathrm{Tr}\Big( A \longrightarrow A; \ a \mapsto S^2(a_{(1)}) \langle p, a_{(2)} \rangle \Big) = 0
  \end{equation}
  for all $p \in A^{\ast}$, since the Hopf algebra $A$ is not cosemisimple and thus $\langle \lambda, 1 \rangle = 0$.

  Now, let $\boldsymbol{b} \in B_n$ be a braid. 
By \eqref{eq4-4}, $\underline{\boldsymbol{b}\, \text{-}\dim }^l ({}_{D(A)}A) = \underline{\rm Tr}_c^l(\widetilde{f})$, where
  \begin{equation*}
    \widetilde{f}
    = (\overbrace{\underline{\mathrm{Tr}}_c^{l,1} \circ \dotsb \circ \underline{\mathrm{Tr}}_c^{l,1}}^{n-1})(\rho(b)).
  \end{equation*}
  Let $f: A_{\rm ad} \longrightarrow \boldsymbol{k}$ be the $A$-linear map corresponding to $\widetilde{f}$ under the isomorphism
  \begin{equation*}
    \begin{CD}
      {\rm End}_{D(A)}({}_{D(A)}A)
      @>{\cong}>> {\rm Hom}_{D(A)}({}_{D(A)}A, I_A(\boldsymbol{k}))
      @>{\cong}>> {\rm Hom}_A(A_{\rm ad}, \boldsymbol{k})      
    \end{CD}
  \end{equation*}
  given by \eqref{eq2-9} and Proposition~\ref{2-3}. 
Then we have $\widetilde{f}(a) = \sum \langle f, a_{(2)} \rangle \, a_{(1)}$ for all $a \in {}_{D(A)}A$, and therefore $\underline{\boldsymbol{b}\, \text{-}\dim }^l ({}_{D(A)}A)$ is equal to the trace of the linear map
  \begin{equation*}
    A \longrightarrow A;
    \quad a \mapsto u \rightharpoonup \widetilde{f}(a)
    = \sum S^2(a_{(1)}) \langle \alpha^{-1}, a_{(2)} \rangle \langle f, a_{(3)} \rangle 
    \quad (a \in A).
  \end{equation*}
  Hence, $\underline{\boldsymbol{b}\, \text{-}\dim }^l ({}_{D(A)}A) = 0$ by \eqref{eq:trace-formula}. 
The equation $\underline{\boldsymbol{b}\, \text{-}\dim }^r ({}_{D(A)}A) = 0$ is proved in a similar way.
\end{proof}

\par \smallskip 
By this theorem, we could say that the braided dimension of the Schr\"{o}dinger module is not interesting as a monoidal Morita invariant for non-cosemisimple Hopf algebras. 
However, the endomorphism of ${}_{D(A)}A$ induced by a braid, such as $\widetilde{f}$ in the above proof, is not generally zero, and thus may have some information about $A$. 
For example, let us consider the map
\begin{equation}
  z_M := \underline{\rm Tr}^{r,1}_c(\rho_M(\sigma_1)): M \longrightarrow M
\end{equation}
for finite-dimensional $M \in {}_{D(A)}\boldsymbol{\sf M}$, where $\rho_M: B_2 \longrightarrow {\rm GL}(M^{\otimes 2})$ is the action of $B_2$. 
One can check that $z_M$ is given by the action of $z := u S(u)$ on $M$. 
Hence, if $M = {}_{D(A)}A$, then
\begin{equation}
  z_{M}(a) = z \bullet a = a_{(1)} \langle \alpha^{-1}, a_{(2)} \rangle
  \quad (a \in {}_{D(A)}A)
\end{equation}
by Lemma~\ref{lem:Sch-mod-u-Su-action}, where $\alpha \in A^{\ast}$ is the distinguished grouplike element. 
Therefore this map has the following information: $z_M$ for $M = {}_{D(A)}A$ is the identity if and only if $A$ is unimodular.

\par \bigskip \noindent 
{\bf \large Acknowledgements}
\par \ 
The authors would like to thank Professor Akira Masuoka for showing a direction to solve our problem and for continuous  encouragement.  
For this research the first author (K.S.) is supported by Grant-in-Aid for JSPS Fellows (24$\cdot$3606), and the second author (M.W.) is partially supported by Grant-in-Aid for Scientific Research (No. 22540058), JSPS.

\par \medskip 
{\sc \small Graduate School of Mathematics, Nagoya University, Furo-cho, Chikusa-ku, Nagoya
464-8602, Japan}
\par 
{\small E-mail address: x12005i@math.nagoya-u.ac.jp} 
\par \medskip 
{\sc \small Department of Mathematics, Faculty of Engineering Science, Kansai University, Suita-shi, Osaka 564-8680, Japan} 
\par 
{\small E-mail address: wakui@kansai-u.ac.jp}

\end{document}